\documentclass{conm-p-l}

\usepackage{amssymb}

\usepackage{graphicx}

\usepackage{latexsym,amsmath,amsthm,amsfonts,epsfig,latexsym,color,enumerate,tikz,epsf}
\usetikzlibrary{cd}

\copyrightinfo{2009}{American Mathematical Society}

\newtheorem{thm}{Theorem}[section]

\newtheorem{cor}[thm]{Corollary}
\newtheorem{prop}[thm]{Proposition}
\newtheorem*{main}{Main Theorem}{}

\newtheorem{lemma}[thm]{Lemma}

\theoremstyle{definition}
\newtheorem{defn}[thm]{Definition}
\newtheorem{example}[thm]{Example}

\theoremstyle{remark}

\numberwithin{equation}{section}

  \def\R{{\mathbb R}}

  \def\Z{{\mathbb Z}}
  \def\N{{\mathbb N}}

    \def\T{{\mathcal T}}

  \def\F{{\mathcal F}}

  \def\diam{\operatorname{diam}}

\def\cK{{\mathcal K}}
  
  \def\sig{\sigma}

  \renewcommand{\>}{\rangle}
  
  \newcommand{\al}{\alpha}
\newcommand{\from}{\colon}

  \newcommand{\p}{\partial}  
\newcommand{\maps}{Maps}

  \newcommand{\cR}{\mathcal{R}}

\begin{document}

\title[Big Out]{Groups of proper homotopy equivalences of
graphs and Nielsen Realization}


\author{Yael Algom-Kfir}
\address{Department of Mathematics, University of Haifa, 199 Abba Khoushy Avenue
Mount Carmel,
Haifa, 3498838
Israel}
\curraddr{}
\email{yalgom@univ.haifa.ac.il}
\thanks{}

\author{Mladen Bestvina}
\address{Department of Mathematics, University of Utah, 155 S 1400 E, RM 233
Salt Lake City, UT 84112, USA}
\curraddr{}
\email{mladen.bestvina@utah.edu}
\thanks{The second author gratefully acknowledges the support by the National Science Foundation
under the grant number DMS-1905720}


\date{November 2022}

\begin{abstract}
    For a locally finite connected graph $X$ we consider the
  group $\maps(X)$ of proper homotopy equivalences of $X$. We show
  that it has a natural Polish group topology, and we propose these groups as an
  analog of big mapping class groups. We prove the Nielsen Realization
  theorem: if $H$ is a compact subgroup of $\maps(X)$ then $X$ is
  proper homotopy equivalent to a graph $Y$ so that $H$ is realized by
  simplicial isomorphisms of $Y$.
\end{abstract}

\maketitle

\section{Introduction}

The group $Out(F_n)$ of outer automorphisms of the free group of rank
$n$ can be thought of as the group of homotopy
equivalences of a finite graph $X$ with $\pi_1(X)\cong F_n$, up to
homotopy. 
In this paper we begin the study of the analogous group
associated with a {\it locally} finite graph $X$. 

\begin{defn} Let $X$ be a locally finite connected graph.
  The {\it mapping class
    group} $\maps(X)$ of $X$ is the group of proper homotopy
  equivalences of $X$, up to proper homotopy.
\end{defn}

Recall that $f:X\to Y$ is a {\it proper homotopy equivalence} if it is
proper and there is a proper map $g:Y\to X$ such that both $fg$ and
$gf$ are properly homotopic to the identity. For an example of a
proper map which is homotopy equivalence but not a proper homotopy
equivalence see Example \ref{noninvertible}.

We will equip $\maps(X)$ with a natural topology which will make it a
Polish group (recall that this means that the underlying topological
space is separable and admits a complete metric). See Section \ref{topology}.
We thus propose $\maps(X)$ as the ``big $Out(F_n)$'' equivalent of
mapping class groups of surfaces of infinite type (or ``big mapping
class groups''), for a survey of the subject see \cite{survey}.
Comparison with mapping class groups has shown to be
very useful in the study of $Out(F_n)$, and we expect that comparison
between $\maps(X)$ and big mapping class groups will likewise prove
fruitful. We remark that the group of all automorphisms
$Aut(F_\infty)$ of the free group of countable rank has a natural
structure of a Polish group (e.g. it is a closed subgroup of the
Polish group of all permutations of $F_\infty$), and so does
$Out(F_\infty)$ since the group of inner automorphisms is
discrete. However, the groups $\maps(X)$ seem more appealing as they
have a more topological flavor and come in great variety since they
depend on the graph $X$. Even when $X$ is a tree the group $Maps(X)$
is of interest, it coincides with the group $Homeo(\p X)$ of
homeomorphisms of the space of ends $\p X$ of $X$ (see Corollary \ref{sigma}).
Note that if $h:X\to Y$ is a proper homotopy equivalence with inverse
$h':Y\to X$, then $f\mapsto hfh'$ induces an isomorphism $Maps(X)\to
Maps(Y)$, which will turn out to be an isomorphism of topological
groups, see Corollary \ref{topiso}.

In this paper we focus on compact subgroups of $\maps(X)$ and prove
the following version of the Nielsen Realization theorem.

\begin{main}
  Let $H$ be a compact subgroup of $\maps(X)$. Then there is a locally
  finite graph $Y$ proper homotopy equivalent to $X$ so that under the
  induced isomorphism $\maps(X)\cong\maps(Y)$ the group $H$ is
  realized as a group of simplicial isomorphisms of $Y$.
\end{main}

Recall that the original Nielsen Realization theorem for finite type
surfaces with negative Euler characteristic was proved by Kerckhoff
\cite{kerckhoff}, stating that any finite subgroup of the mapping
class group of a surface of negative Euler characteristic can be
realized by isometries of a complete hyperbolic metric with finite
area. The version for $Out(F_n)$, proved in \cite{NR1,NR2,NR3,NR4},
states that a finite subgroup of $Out(F_n)$ can be realized as a group
of simplicial isomorphisms of a finite graph with fundamental group
$F_n$. For big mapping class groups Nielsen Realization was proved
recently by Afton-Calegari-Chen-Lyman \cite{rylee-etal}. Among the
consequences is that compact subgroups of big mapping class groups are
finite. This is not the case for $\maps(X)$. For example, let $X$ be
the graph obtained from $[0,\infty)$ by attaching two loops at every
  integer point. The group of symmetries of this graph is the compact
  group $H^{\infty}= \prod_{i=1}^\infty H$, where $H$ is the group of
  symmetries of order 8 of the wedge of two circles. Note that $X$ is
  proper homotopy equivalent to the graph $Y$ obtained from
  $[0,\infty)$ by attaching three circles at every integer point, and
    the group of symmetries of $Y$ is $G^\infty$, where $G$ is the
    group of order 48 of symmetries of the wedge of three circles. The
    groups $\maps(X)$ and $\maps(Y)$ are isomorphic as topological
    groups, but the realization using different graphs displays
    different compact subgroups.

In addition to big mapping class groups $Mod(\Sigma)$ and groups
$Maps(X)$, the groups $Homeo(Z)$ of homeomorphism groups of compact
totally disconnected metrizable spaces with compact-open topology have
many similar properties and are studied more classically. Excluding
the cases when $\Sigma$ and $X$ have finite type and $Z$ is finite,
all these groups have underlying space homeomorphic to the
irrationals, they all admit clopen subgroups forming a basis of
neighborhoods of the identity, and they all satisfy the Nielsen
Realization theorem (for $Homeo(Z)$ this is proved quickly in Section
\ref{5}, it states that for any compact subgroup $H<Homeo(Z)$ there
is a metric on $Z$ so that $H$ consists of isometries).

Since this paper was first circulated, Domat, Hoganson and Kwak
\cite{DHK} investigated the coarse geometry of the pure subgroup
$PMaps(X)$ of $Maps(X)$. Their work points out the differences between
these three classes of topological groups.

\vskip 1cm

{\bf Plan of the paper.} We start by recalling the Classification
theorem for locally finite graphs in Section
\ref{s:classification}. We also introduce the notation and review the
homotopy extension theorem in our setting and some of its
consequences. In Section \ref{s:algebra} we explore the natural
homomorphism $\maps(X)\to Out(\pi_1(X))$ and in particular we look for
conditions that guarantee that a proper map $X\to X$ is properly
homotopic to the identity. In the case when $X$ is a core graph
(i.e. it is the union of immersed loops) the criterion is particularly
simple: if $f$ induces the identity in $Out(\pi_1(X))$, it is properly
homotopic to the identity. In the other extreme, when $X$ is a tree,
$f$ is properly homotopic to the identity whenever it fixes the end
space $\partial X$. The general case is more complicated since rays
attached to the core graph could wrap around the core, but we show
that if $f$ is identity in $\pi_1$, fixes the ends, and preserves
proper lines, then $f$ is properly homotopic to the identity.

Section \ref{topology} is devoted to defining the topology on
$\maps(X)$ and establishing that it is a Polish group. As long as $X$
is of infinite type, we show that the underlying topological space of
$\maps(X)$ is homeomorphic to the set of irrationals, but of course
the group structure will depend on $X$.

The remainder of the paper is devoted to the proof of the Main
Theorem. This is also divided into cases, with the two extremes of $X$
being a core graph and being a tree discussed first. When $X$ is a
tree, by averaging we find an $H$-invariant metric on the space of
ends $\p X$. From this we construct $H$-invariant finite covers by
disjoint clopen sets that refine each other and with mesh going to
0. The mapping telescope of this sequence is the desired tree $Y$.

The heart of the argument is the case when $X$ is a core graph. For
concreteness imagine that $X$ is the graph obtained from the ray
$[0,\infty)$ by attaching a circle at every integer point. We then
  cover $[0,\infty)$ by large intervals $J_1,J_2,\cdots$ so that
    $J_i\cap J_j=\emptyset$ if $|i-j|>1$ and so that $J_i\cap J_{i+1}$
    are large as well, controlling the properness of elements of
    $H$. Each $J_i$ and $J_{i}\cap J_{i+1}$ defines a subgraph of $X$
    and a free factor of $\pi_1(X)$. By intersecting the
    $H$-translates of these free factors we obtain $H$-invariant free
    factors $F_i^*$ and $F_{i,i+1}^*$ respectively. Using Nielsen
    Realization in finite rank we find finite graphs $\Gamma_{i,i+1}$
    where $H$ acts by simplicial isomorphisms realizing
    $F_{i,i+1}^*$. We then use the Relative Nielsen Realization, due
    to Hensel-Kielak \cite{hensel-kielak}, to construct finite graphs
    $\Gamma_i$ where $H$ acts by simplicial isomorphisms realizing
    $F_i^*$, and that contain $\Gamma_{i-1,i}$ and $\Gamma_{i,i+1}$ as
    disjoint invariant subgraphs. Finally, we glue the $\Gamma_i$'s
    along these subgraphs to obtain $Y$. In general, when $X$ is a
    tree with circles attached at vertices, the above outline still
    works, but instead of free factors we have to consider free factor
    systems which makes the notation a bit more complicated.

    In the general case, we first use the case of core graphs to
    reduce to the situation where $H$ is already acting on the core by
    simplicial isomorphisms. The graph $X$ is obtained from the core
    by attaching trees, and the central part of the proof in this case
    is to see how to attach new trees in an equivariant fashion. The
    new trees are going to be mapping telescopes made of partition
    elements (as in the tree case) subordinate to suitable clopen sets
    in $\p X$. To
    accomplish this we prove a fixed point theorem (see Lemma
    \ref{fixed}) that provides a suitable point in the core where
    these telescopes are attached.

    In order to verify that the action of $H$ on the new graph $Y$ is
    conjugate to the given action on $X$ we use the machinery
    developed in Section \ref{s:algebra}, see Corollary
    \ref{criterion}.

\section{The classification of locally finite connected
  graphs}\label{s:classification}

Let $X$ be a locally finite infinite connected graph. The fundamental
group $\pi_1(X)$ is free and we denote its rank by $g(X)\in
\{0,1,2,\cdots,\infty\}$ and think of it as the ``genus'' of
$X$. Let $$\p X=\underset\longleftarrow\lim_{K\subset
  X}\pi_0(X\smallsetminus K)$$ be the space of ends of $X$ with its
usual inverse limit topology, where the limit runs over all compact
subsets $K\subset X$ . Then $\p X$ is a totally disconnected compact
metrizable space (recall that these are precisely the spaces
homeomorphic to a closed subset of the Cantor set). The union $\hat
X=X\sqcup\p X$ has a natural topology that makes it compact; it is the
{\it Freudenthal} (or {\it end}) compactification of $X$. The basis of
open sets in $\hat X$ consists of open sets in $X$, and for every
compact $K\subset X$ and every
component $U$ of $X\smallsetminus K$ the set $\hat U$ which is the
union of $U$ and the set of ends that map to $U$.  We will sometimes
abuse notation and talk about a neighborhood $U$ in $X$ of an end
$\beta\in\p X$; what we mean is the intersection of such a
neighborhood $\hat U$ in $\hat X$ with $X$. The end compactification
can also be constructed in the same way for connected, locally finite
cell complexes.  Every proper map $f:X\to Y$ between such complexes
extends continuously to a map $\hat X\to\hat Y$ between their end
compactifications. For simplicity we will usually denote this extension, as
well as its restriction $\p X\to\p Y$, by $f$ as well. Properly
homotopic maps induce the same map between the boundaries.

Denote by
$X_g\subset X$ the {\it core} of $X$, i.e. the smallest subgraph that
contains all immersed loops. Thus $X_g=\emptyset$ precisely when $X$
is a tree. Let $\p X_g\subseteq \p X$ be the space of ends of $X_g$;
it is a closed subspace of $\p X$ (and consists of ends ``accumulated
by genus''). Thus $\p X_g=\emptyset$ precisely when $g(X)<\infty$. In
general, $X$ is either a tree or is obtained from $X_g$ by attaching
trees.

\begin{defn}[Characteristic pairs]
If $g(X)<\infty$, its characteristic pair is $(\p X,g(X))$, otherwise
its characteristic pair is $(\p X, \p X_g)$.
\end{defn}

The following is the analog of Ker\'ekj\'art\'o's classification
theorem for surfaces and was proved by
Ayala-Dominguez-Marquez-Quintero \cite{admq}.

\begin{thm}\label{graphClassification}
Let $X, Y$ be locally finite connected graphs. Then a homeomorphism of
characteristic pairs extends to a proper homotopy equivalence. If $X$
and $Y$ are trees the extension is unique up to proper homotopy. 
\end{thm}

In the case when the genus is finite, a homeomorphism of
characteristic pairs means a homeomorphism between the spaces of ends
together with the information that the genera are equal.

If $f:X\to X$ is a proper homotopy equivalence, then the extension
$\partial X\to\p X$ is a homeomorphism which preserves $\p X_g$. Thus
we have a well defined homomorphism
$$\sigma:\maps(X)\to Homeo(\p X,\p X_g)$$
to the group of homeomorphisms of the pair $(\p X,\p X_g)$.
The following is then an immediate corollary of the classification theorem.

\begin{cor}\label{sigma}
  The homomorphism $\sigma$ is always surjective.
  If $X$ is a tree the map $\sigma:\maps(X)\to Homeo(\partial X)$ is
  an isomorphism.
\end{cor}

In light of this we will usually focus on the kernel of $\sigma$,
which is the {\it pure mapping class group} of $X$, and we denote it
by $P\maps(X)$.

\begin{defn}
The pure group $P\maps(X)$ is the subgroup of $\maps(X)$ consisting of
$f \in \maps(X)$ so that $f:\p X\to\p X$ is the identity.
\end{defn}

We shall make use of the following concepts and lemmas from \cite{admq}.

\begin{defn}[Standard Models]
The Cantor tree $T$ is the rooted binary tree embedded in the plane so
that its boundary is the standard trinary Cantor set in $[0,1] \times
\{0\}$.  For each closed subset $B$ of the Cantor set, let $T_B$ be
the union of the set of rays in $T$ initiating at the root and terminating in $B$. For a characteristic pair $(B, g)$ where $B$ is a closed subset of the Cantor set and $g$ is a natural number let $X_{(B,g)}$ be the tree $T_B$ with $g$ loops attached at the root. The characteristic set of $X_{(B,g)}$ is $(B,g)$. 
For two closed non-empty subsets $A \subseteq B$ of the Cantor set, let $X_{(B,A)}$ be the tree $T_B$ with a one edge loop attached at each vertex of the subtree $T_A$. Again, the characteristic set of $T_{(B,A)}$ is $(B,A)$. These trees are called the Standard Models. 
\end{defn}

\begin{figure}[t]
  \setlength{\captionindent}{0pt}
\centering
\begin{minipage}[t]{0.3\linewidth}
\includegraphics[width=\textwidth]{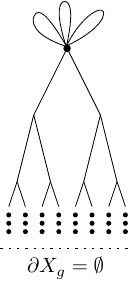}
\caption{A graph with a finitely generated fundamental group. In this case $\partial X_g = \emptyset$ and the space of ends can be any closed subset of the Cantor set. }
\label{fig:minipage1}
\end{minipage}
\quad
\begin{minipage}[t]{0.3\linewidth}
\includegraphics[width=\textwidth]{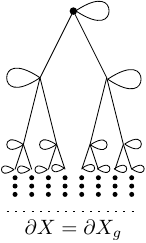}
\caption{This is a core graph, $\partial X= \partial X_g$. By deleting a part of the Cantor tree we get a standard model for $(A,A)$ where $A$ is any closed subspace of the Cantor set. }
\label{fig:minipage2}
\end{minipage}
\quad
\begin{minipage}[t]{0.3\linewidth}
\includegraphics[width=\textwidth]{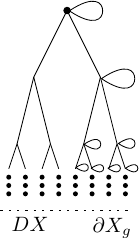}
\caption{In this case both $\partial X_g$ and its complement
  $DX:=\partial X\smallsetminus \partial X_g$ are non-empty. By deleting loops/subtrees from the space in Figure \ref{fig:minipage2} we get a model for any pair of closed subsets $\partial X = A \supset B = \partial X_g$ of the Cantor set.}
\label{fig:minipage3}
\end{minipage}
\end{figure}

\begin{cor}
Every locally finite connected infinite graph is proper homotopy equivalent to a Standard Model. 
\end{cor}

In particular, we can assume that $X$ has no valence 1 vertices, and
that $X=X_g$ if $\p X_g=\p X$.

The following is the standard Homotopy Extension Theorem, see
\cite{hatcher}. We will use it often, usually without saying it. Given a subgraph $Y \subset X$, the frontier of $Y$ is the set of vertices in $Y - int(Y)$ denoted $Fr(Y)$. 

\begin{prop}\label{pasting_homotopies}
Let $Y$ be a subgraph of $X$, let $H \from Y \times I \to X$ be a
proper homotopy from
$h$ to $f$. Let $u \from X \to X$ be a proper map so that $u|_Y =
h$. Then there is a proper homotopy $H' \from X\times I \to X$
extending $H$ and $u$, i.e. $H'(x,0)=u(x)$.
Moreover, if $K \subset X$ is a subset so that $H(Fr(Y) \times I)
\cap K = \emptyset$ and $u(\overline{X\smallsetminus Y}) \cap K =
\emptyset$ then $H'((X\smallsetminus Y) \times I) \cap K = \emptyset$. 
\end{prop}

\begin{proof}
  First define the extension on the vertices $v$ of $X$ outside of $Y$
  by $H'(v,t)=u(v)$. Now let $e$ be an edge with endpoints $a,b$. If
  both $a,b$ are outside of $Y$ define $H'$ on $e\times I$ to be
  stationary as well: $H'(x,t)=u(x)$ for $x\in e$. Otherwise let
  $\alpha:[0,1]\to e$ be a parametrization of $e$.  Let $P$ be a
  retraction from $I \times I$ to $(\{0\} \times I) \cup (I \times
  \{0\}) \cup (\{1\} \times I)$. Then $e\times I$ can be identified
  with $I\times I$ (with $0\times I$ and $1\times I$ identified if $e$
  is a loop). The homotopy is already defined on $(\{0\} \times I)
  \cup (I \times \{0\}) \cup (\{1\} \times I)$ and we extend it to
  $I\times I$ by composing with $P$.

  We leave the verification that $H'$ is proper and the last sentence
  to the reader.
\end{proof}

We will also sometimes have a need to restrict $f\in Maps(X)$ to an ``invariant''
subgraph $Y\subset X$ e.g. $Y=X_g$. Since $f$ is defined only up to proper
homotopy, $Y$ will usually not satisfy $f(Y)\subseteq Y$, but this
will be true after a proper homotopy $H$. We will also want to know that
the proper homotopy class of the restriction is independent of the
choice of $H$. A bad example to keep in mind is the projection
$\pi:\R^2\to\R$. There are proper maps $f:\R\to\R^2$ such that $\pi f$
is not proper (and there are analogous examples with Cayley graphs
$Cay(\Z)\to Cay(\Z^2)$), and there are pairs of properly homotopic maps
$f,g:\R\to\R^2$ such that $\pi f$ and $\pi g$ are proper, but not
properly homotopic.

\begin{lemma}\label{restriction}
  Let $Y$ be a subgraph of $X$ such that inclusion $Y\hookrightarrow
  X$ induces an injection $\partial Y\hookrightarrow\partial X$. Let
  $\pi:X\to Y$ be a retraction with the following property: for every
  $\beta\in\partial Y$ and every neighborhood $U$ of $\beta$ in $Y$
  there is a neighborhood $V$ of $\beta$ in $X$ so that
  $\pi(V)\subseteq U$. Let $Z$ be a locally compact metrizable
  space. 
  \begin{enumerate}[(i)]
    \item If $f:Z\to X$ is a proper map such that $f(\partial
      Z)\subseteq \partial Y$ then $\pi f:Z\to Y$ is proper.
      \item If $f,g:Z\to Y$ are proper maps such that they are
        properly homotopic within $X$, then they are properly
        homotopic within $Y$.
  \end{enumerate}
\end{lemma}

\begin{proof}
  To prove (i), note that if $\gamma\in \partial Z$ then
  $f(\gamma)=\beta\in\partial Y\subseteq\partial X$. If $U$ is a
  neighborhood of $\beta$ in $Y$, there is a neighborhood $V$ of
  $\beta$ in $X$ so that $\pi(V)\subseteq U$. Since $f$ is proper,
  there is a neighborhood $W$ of $\gamma$ in $Z$ so that
  $f(W)\subseteq V$. It follows that $\pi f(W)\subseteq U$, so $\pi f$
  is proper.
  For (ii), apply (i) to
  a proper homotopy $H:Z\times I\to X$. The fact that $H(\partial
  (Z\times I))\subseteq \partial Y$ follows from the assumption that
  $\partial Y\subseteq \partial X$. 
\end{proof}

\begin{prop}\label{extension}
  Suppose $X,Y$ are two locally finite connected graphs, $f:X_g\to Y$
  a proper map that induces $\overline f:\p X_g\to\p Y$ and let
  $\overline F:\p X\to\p Y$ be an extension of $\overline f$.  Then
  there is a proper map $F:X\to Y$ that extends $f$ and induces
  $\overline F$.
\end{prop}

\begin{proof}
First consider the case when $X$ and $Y$ are trees, and choose base
vertices $x_0\in X$, $y_0\in Y$. We are given $\overline F:\p X\to\p Y$
and we have to construct a proper map $F:X\to Y$. If $\overline F$ is a
homeomorphism the existence of $F$ follows from the Classification
theorem (and in fact it is unique up to proper homotopy). In general,
we can construct $F$ as follows. When $v$ is a vertex of $X$ let the
shadow $Sh_X(v)\subseteq \p X$ be the set of endpoints of rays that
start at $x_0$ and pass through $v$. When $A\subseteq \p X$ contains
at least two points, let $\sup A$ be the vertex $v\in X$ with the
largest distance $|v|$ from $x_0$ satisfying $Sh_X(v)\supseteq A$. If
$A=\{\beta\}\subseteq \p X$ is a single point, define $\sup A$ to be
$\beta$ and let $|\beta|=\infty$. Make the similar definition for
subsets of $\p Y$. For a vertex $v\in X$ consider $\sup
\overline F(Sh_X(v))$. If this is a vertex $w$ at distance $|w|\leq |v|$
from $y_0$ then define $F(v)=w$, and otherwise define $F(v)$ as the
vertex at distance $|v|$ from $y_0$ along the segment (or ray)
$[y_0,w]$. Extend $F$ linearly to the edges of $X$. 

Now consider the general case. We may assume that $f$ sends vertices
to vertices. First suppose that there is a maximal tree $T\subseteq Y$
such that $\p T=\p Y$. Since $X$ is simplicial and locally finite,
$X\smallsetminus X_g$ is a countable (or finite) union of trees, and let
$T_i$ for $i \in \N$ be the closure of a component of $X\smallsetminus
X_g$. We
denote by $x_i$ the point of intersection of $T_i$ and $X_g$, so $x_i$ is
the root of $T_i$. For each $i \in \N$ let $F_i \from T_i \to
T\subseteq Y$ be the map constructed in the first paragraph so that
$F_i(x_i) = f(x_i)$ and $\partial F_i = \overline F|_{\p T_i}$.  We
define $F$ by gluing the maps $f$ and $F_i$ for all $i \in \N$.

One way to avoid constructing a special maximal tree is as
follows. Let $Z$ be a Standard Model proper homotopy equivalent to
$Y$. Then the underlying tree in $Z$ has the same ends. So we may
apply the above paragraph to the composition $X_g\to Y\to Z$ and get
an extension $X\to Z$, which we then compose with the inverse proper homotopy
equivalence $Z\to Y$ to get $F:X\to Y$. The map $F$ may not agree with
$f$ on $X_g$ but it is properly homotopic to it, so we conclude by
applying the Homotopy Extension Theorem (Proposition
\ref{pasting_homotopies}).
\end{proof}

\section{Relationship with $Out(\pi_1(X))$.} \label{s:algebra}

In this section we will investigate the relationship between $\maps(X)$
and $Out(\pi_1(X))$. First, there is a natural homomorphism
$$\Psi:\maps(X)\to Out(\pi_1(X))$$ that sends $h\in \maps(X)$ to
(the outer automophism class of) $h_*:\pi_1(X)\to\pi_1(X)$. 
If the
genus of $X$ is infinite, $\Psi$ is not onto since there are
automorphisms not realized by proper maps. On the other hand, $\Psi$
is onto when $g(X)<\infty$. We will show first that $\Psi$ is
injective if $X$ is a core graph (meaning $X=X_g$), and in general we
will describe the kernel of $\Psi$. The next theorem is analogous
to the fact that a homeomorphism of a surface with nonabelian
fundamental group that induces identity in $\pi_1$ is isotopic to the
identity, see \cite{epstein,Hernandez-Valdez-Morales}.

\begin{thm}\label{id}
Suppose $X$ is a core graph and let $f \from X \to X$ be a proper map
so that $f_*=id\in Out(\pi_1(X))$.
Then $f$ is properly homotopic to the identity on $X$. In other words,
$\Psi$ is injective.
\end{thm}

\begin{proof}
In the proof we will use the fundamental property of graphs that
disjoint nontrivial loops are not homotopic and that nullhomotopic
loops can be nullhomotoped within their images.  We may assume that
$X$ is a Standard Model. Note that $f$ necessarily induces the
identity on the space of ends. Indeed, if $f(\beta)\neq \beta$ for an
end $\beta$, there will be an immersed loop $\alpha$ in $X$ near $\beta$ such
that $f(\alpha)$ is disjoint from $\alpha$, and in particular $f$ does
not fix the conjugacy class of $\alpha$.

Next, we can assume, by applying a proper homotopy (using Proposition
\ref{pasting_homotopies}) that $f$ fixes all
vertices and moreover, by homotoping the root $v$ around a loop, that
$f_*:\pi_1(X,v)\to \pi_1(X,v)$ is the identity. 

We will now construct a proper homotopy between the identity and $f$. If
$w$ is a vertex, let $e_1e_2\cdots e_k$ be the edge path in the
underlying tree $T$ from $v$ to $w$ and define $H:\{w\}\times I\to X$
to be the tightened path $\bar e_k \dots \bar e_1 f(e_1) \dots
f(e_k)$. Also define $H$ on $X\times \{0\}$ to be identity and on
$X\times\{1\}$ to be $f$ (see Figure \ref{homotopyFig}). We will argue below that $H$ defined so far
is proper. If $e$ is an edge in $T$ then $H$ is defined on
$\partial(e\times I)$ and is nullhomotopic on this loop. Thus we may
extend $H$ to all such 2-cells keeping the image contained in the
image of $\partial (e\times I)$. Finally, $H$ extends to the cylinders
$x_w\times I$, where $x_w$ is the loop attached at $w$, using the fact
that $f(e_1\cdots e_k x_w \overline e_k\cdots \overline e_1)\simeq
e_1\cdots e_k x_w \overline e_k\cdots \overline e_1$. We again ensure
that the image of the extension is contained in the image of the
boundary of the 2-cell, so the extended $H$ will be proper.

\begin{figure}[h]
\begin{center}
\includegraphics{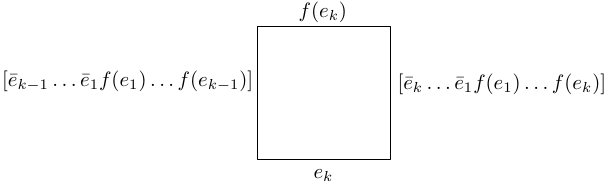}
\caption{The homotopy $H$. The brackets signify that we take the immersed path homotopic to the given one rel endpoints. \label{homotopyFig}}
\end{center}
\end{figure}

It remains to show that $H$ defined on the 1-skeleton is proper. This
is where we will use the assumption that $X$ is a core graph. Let
$K\subset X$ be a finite subgraph. Since $f$ is proper, there is a
finite subgraph $L\subset X$ such that $f(X\smallsetminus L)\subseteq
X\smallsetminus K$. Let $w$ be a vertex outside of $L$, $e_1\cdots
e_k$ the edge path from $v$ to $w$ in $T$, and $x_w$ the loop attached
to $w$. Thus $f(x_w)\cap K=\emptyset$. From the fact that
$$f(e_1 \dots e_k) f(x_w) f(\bar e_k \dots \bar e_1) \simeq e_1 \dots
e_k \cdot x_w \cdot \bar e_k \dots \bar e_1$$ we see that after
tightening $f(e_1 \dots e_k)$ does not cross any loops attached to
vertices in $K$. For example, the fundamental group can be thought of
as the free group on the attached loops, and if $y$ is the last loop
in $K$ crossed by $[f(e_1\cdots e_k)]$ the word $[f(e_1\cdots e_k)]\cdot
[f(x_w)]\cdot [f(\bar e_k\cdots \bar e_1)]$ could be tightened by
tightening the portion between the corresponding $y$ and $\bar y$, and
would not yield the trivial word.
Therefore $H(\{w\}\times I)$ is
disjoint from $K$.
\end{proof}

The fundamental group of $X$ does not ``see'' the ends of $X$ not
accumulated by genus. For example, if $X$ is a tree the mapping class
group $\maps(X)$ is isomorphic to the homeomorphism group
$Homeo(\partial X)$ and may be quite nontrivial, while
$\pi_1(X)=1$. It is therefore natural, when studying the kernel of
$\Psi$, to restrict to the pure mapping class group $P\maps(X)$.

We will now describe the kernel of the restriction
$$\Psi_P:P\maps(X)\to Out(\pi_1(X))$$ of $\Psi$. This is well-known
when $X$ is obtained from a finite graph, say of rank $r$ so that
$\pi_1(X)=F_r$, by attaching a finite number, say $n$, of rays. These
rays can be thought of equivalently as distinguished points in the
finite graph. When $n=1$ we have $\maps(X)\cong Aut(F_r)$ and when
$n>1$ then $P\maps(X)\cong Aut(F_r)\ltimes F_r^{n-1}$ with the natural
diagonal action of $Aut(F_r)$ on $F_r^{n-1}$. The $F_r^{n-1}$ factor
can be thought of as measuring the marking of $(n-1)$ 
distinguished points with respect to the remaining distinguished
point, which is considered to be the basepoint. This will be
generalized in Corollary \ref{semidirect}. When $r\geq 2$ the kernel
of $\Psi_P$ is isomorphic to $F_r^{n-1}$ 
 represented by maps that are
identity on the finite graph and send each ray $R$ to a ray of the
form $w_R R$ for some loop $w_R$ in the finite graph.

Let $X$ be a locally finite graph and we assume $DX := \p X - \p X_g
\neq \emptyset$ and choose $\al_0 \in DX$. This will be the basepoint
``at infinity''.  Let $\pi_1(X,\al_0)$ be the set of proper homotopy
classes of lines $\sig \from \R \to X$ so that $\lim_{t \to -\infty}
\sig(t) = \lim_{t \to \infty} \sig(t) = \al_0$, with concatenation as
the group operation. Notice that
concatenation makes sense since $\al_0 \in DX$ and any two rays
limiting to $\al_0$ eventually coincide, up to a proper homotopy. Given 
$x_0\in X$, there is an isomorphism
\begin{equation*}\label{pi1eq} \pi_1(X,x_0) \to
  \pi_1(X,\al_0)\end{equation*} given by $\gamma
\to \bar \rho_{x_0} \gamma \rho_{x_0}$ where $\rho_{x_0}$ is a fixed ray in $X$ from $x_0$ to $\al_0$. 
Moreover, if $f \in \maps(X)$ fixes $\al_0$ then $f$ induces a map $f_0
\in Aut(\pi_1(X,\al_0))$. 

We first consider the case of the graph $X=X_g^*$ obtained from a core
graph (or a point) $X_g$ by attaching a single ray.

\begin{lemma}\label{Ystar}
  For $X=X_g^*$ the kernel of $\Psi:\maps(X)\to Out(\pi_1X)$ (or
  $\Psi_P:P\maps(X)\to Out(\pi_1X)$) is isomorphic to
  $\pi_1(X)=\pi_1(X_g)$ when this group is nonabelian, and otherwise
  it is trivial.
\end{lemma}

When the genus $n=g(X)$ is finite, the lemma says that the kernel
$Aut(F_n)\to Out(F_n)$ is $F_n$ for $n>1$ and otherwise it is
trivial. 

\begin{proof}
  Let $f\in \maps(X)$ induce identity in $Out(\pi_1(X))$. Using Lemma
  \ref{restriction} applied to the nearest point projection $\pi:X\to
  X_g$ we see that after a
  proper homotopy we may assume that $f$ preserves the core $X_g$, and
  thus by Theorem \ref{id} we may assume $f$ is identity on
  $X_g$.
Let $\rho_0$ denote the geodesic ray in $X$ that intersects $X_g$ at one point and such that $X = X_g \cup \rho_0$. Let $c(f)$ be the homotopy class of $\bar \rho_0 f(\rho_0)$ in $\pi_1(X, \al_0)$. 
Then $f_0 \in Aut(\pi_1(X, \al_0))$ is just conjugation by $c(f)$.  
  The map $f\mapsto c(f)$ is a homomorphism $Ker(\Psi)\to
  \pi_1(X,\alpha_0)$. When $\pi_1(X)$ is non-abelian it is an isomorphism. 
If $X_g$ is a circle, $f$ can be homotoped to the
  identity by a homotopy that rotates the circle to unwind the
  attached ray, so $Ker(\Psi)$ is trivial.
\end{proof}

We now consider the general case. Let $X$ be a Standard Model which is
not a tree, let $X_g$ be the core subgraph of $X$, and we assume $DX
\neq \emptyset$.  Fix $\al_0 \in DX$ and let $\rho_0$ be the ray in
$X$ intersecting $X_g$ in a point and limiting to $\al_0$.  Let
$T\subset X$ be the underlying tree, and let $T_g=X_g\cap T$ be the
underlying tree in the core, and likewise let $T_g^*=T_g\cup\rho_0$ be
the underlying tree in $X_g^*=X_g\cup\rho_0$. 
Thus $X_g\subset X_g^*\subseteq X$ and both inclusions
are homotopy equivalences.  We note that restriction maps
$$P\maps(X)\to P\maps(X_g^*)\to P\maps(X_g)$$
are well-defined by Lemma \ref{restriction}, where for the retraction
$\pi$ we take the nearest point projection. In fact, we have a
factorization of $PMaps(X)\to Out(\pi_1(X))$ as
$$P\maps(X)\to P\maps(X_g^*)\to P\maps(X_g)\to Out(\pi_1(X))$$

Our next goal is to describe $Ker(P\maps(X)\to P\maps(X_g^*))$. To
that end, we introduce the group $\cR$ that measures how the rays
towards the ends in $DX$ wrap around the loops in the core graph $X_g$.

\begin{defn}
The group
$\cR$ as a set is the collection of functions $h \from DX \to \pi_1(X_g^*,
\al_0)$ satisfying 
\begin{itemize}
  \item[(R0)] $h(\al_0)=1$.
  \item[(R1)] $h$ is locally constant.
\item[(R2)] For all $\beta \in \p X_g$ and every neighborhood
  $U$ of $\beta$ in $X$ there exists a ray $\rho_U$
  from a point in $U$ to $\alpha_0$ which is the concatenation of a
    segment in $T$ and $\rho_0$ so that if $(\beta_i)_{i=1}^\infty
  \subset DX$ limits to $\beta$ then for large enough $i$,
  $h(\beta_i)= \bar\rho_U*\gamma_i*\rho_U$ where $\gamma_i$ is a loop
  contained in $U$.
\end{itemize}
The group operation in $\cR$ is pointwise multiplication in
$\pi_1(X,\al_0)$.
\end{defn}

\begin{defn}\label{defOfPhi}
We start by assigning an element $\Phi_T(f)\in\cR$ to certain proper
maps $f:X\to X$. More precisely assume:
\begin{enumerate}[(i)]
\item $f:\p X\to\p X$ is the identity, and
  \item either $DX$ is compact or
    $f_*:\pi_1(X,\al_0)\to\pi_1(X,\al_0)$ is the identity.
\end{enumerate}

Note that we do not assume that $f$ is a proper homotopy equivalence,
cf. Example \ref{noninvertible}.

If $\beta\in DX\smallsetminus\{\alpha_0\}$ 
let $\ell_\beta$ be the bi-infinite line in $T$ connecting $\al_0$ to
$\beta$.
Define the map 
\[ \begin{array}{l}
\Phi_T(f)=h \from DX \to \pi_1(X,\al_0) \\[0.2 cm]
h(\beta) = \left\{ 
\begin{array}{ll}
\ell_\beta f(\bar\ell_\beta) & \beta \in DX-\{ \al_0 \}, \\
1 & \beta = \al_0 
\end{array} \right.
\end{array} \]
\end{defn}

We claim that indeed $h \in \cR$. That $h$ is locally constant follows from the
observation that if $\beta\in DX$ there is a neighborhood $U\subset DX$ of
$\beta$ so that if $\ell$ is a line joining two distinct points in $U$
then $f(\ell)$ is properly homotopic to $\ell$.

Condition (R2) is vacuous if $DX$ is compact so suppose $f_*=id$. Thus
we may assume that $f$ is identity on $X_g^*$. Let $\ell_\beta$ be the
line in $T$ from $\al_0$ to $\beta$, and similarly let
$\ell_{\beta_i}$ be the line in $T$ from $\al_0$ to $\beta_i$.  Note
that the lines $\ell_{\beta_i}$ converge $\ell_\beta$. The map $f$
fixes $\ell_\beta$ and thus for large $i$ takes $\ell_{\beta_i}$ to a
line that agrees with $\ell_{\beta_i}$ outside a given neighborhood
$U$ of $\beta$. This proves (R2).

\begin{thm}\label{kernel}
  Assume $f \colon X\to X$ is proper, $\al_0\in DX$, and
  $f_*=id \from \pi_1(X,\alpha_0)\to\pi_1(X,\al_0)$. If $f \from \p X\to\p X$ is
  identity and
  $\Phi_T(f)=1$
  then $f$ is properly homotopic to the identity. Moreover,
  $$\Phi_T:Ker(P\maps(X)\to P\maps(X_g^*))\to\cR$$ is an isomorphism.
\end{thm}

\begin{proof}
Let $\cK=Ker(P\maps(X)\to
P\maps(X_g^*))$. 
We start by arguing that
$\Phi_T \from \cK \to \cR$ is a homomorphism. Let $f,g\in\cK$, and we
assume that $f,g$ are identity on $X_g^*$. Let $\ell_\beta$ be the line
as in the definition of $\Phi_T$. We have
\[
\Phi_T(gf)(\beta)=\ell_\beta \cdot gf(\bar\ell_\beta)=\ell_\beta\cdot
g(\bar\ell_\beta) \cdot g(\ell_\beta \cdot
f(\bar\ell_\beta))=\Phi_T(g)(\beta)\cdot 
g(\Phi_T(f)(\beta))
\]
which equals $\Phi_T(g)(\beta)\cdot \Phi_T(f)(\beta)$
since $g$ acts as the identity on $\pi_1(X,\alpha_0)$.

We next show that if $\Phi_T(f)=1$, then $f\simeq id$. We may assume
$f$ is identity on $X_g^*$. Consider the universal cover $\tilde X$ of
$X$ and let $\tilde X_g$ be the preimage of $X_g$ to $\tilde X$ (which is
connected).  Let $\tilde f$ be the lift of $f$ that restricts to the
identity on $\tilde X_g$. The assumption that $\Phi_T(f)(\beta)=1$ for
every $\beta\in DX$ amounts to saying that $\tilde f$ fixes the ends
of $\tilde X$.  The straight line homotopy $\tilde H$ from $\tilde f$
to $id$ is equivariant with respect to the deck group and descends to
the homotopy $H \from X \times I \to X$ from $f$ to $id$. It remains
to show that $H$ is proper. It is useful to describe $H$ directly. If
$x\in X$ consider a ray $\rho_x$ from $x$ to $\alpha_0$ in $X$. The
path $H(\{x\}\times I)$ is the tightened path $f(\rho_x) \bar
\rho_x$.

Let $K$ be a compact set in $X$.  As in the proof of Theorem \ref{id}
it is enough to show that there is a compact set $S$ so that for each
vertex $v$ outside $S$, $H(\{v\} \times I) \cap K = \emptyset$.

Assume $x\in X$ is in a small neighborhood of an end
$\beta\in \partial X$. If $\beta\in\partial X_g$ then $f(\rho_x)$ and
$\rho_x$ agree outside a bit larger neighborhood by the assumption
that $f$ fixes $\beta$ and $X_g^*$, so the path $H(\{x\}\times I)$ is
also close to $\beta$. If $\beta\in DX$ consider the concatenation
$\bar\rho_x \gamma$ where $\gamma$ is the ray from $x$ to $\beta$ in
$T$. This concatenation is properly homotopic to the line $\ell_\beta$
from the definition of $\Phi_T$, and so by our assumption that
$\Phi_T(f)=1$ we have
$f(\bar\rho_x \gamma)\simeq \bar\rho_x\gamma$. Since $f$ fixes
$\beta$, $f(\gamma)$ is in a bit larger neighborhood of $\gamma$, so
$f(\rho_x)\bar\rho_x\simeq f(\gamma)\bar\gamma$ is also close to
$\beta$.

Finally, we argue that $\Phi_T:\cK\to\cR$ is onto. Let $h\in\cR$.  We
define a proper homotopy equivalence $f \from X \to X$.  Let
$f|_{X_g^*} = id$. Let $S_w$ be the tree attached to a vertex $w\in
X_g^*$, i.e. the closure of a component of $T\smallsetminus
T_g^*$. Then $\partial S_w$ is compact and from the fact that $h$ is
locally constant we see that there is some distance $C$ so that for
every edge $e\subset S_w$ at distance $C$ from $w$, $h$ is constant on
the ends of the unbounded component of $S\smallsetminus e$. Define
$f|S_w$ to be identity on all edges of $S$ other than those at
distance $C$ from $w$. On such an edge $e$ define $f$ to be the
immersed path with the same endpoints as $e$ and so that if
$\ell_\beta$ is a line that crosses $e$ then $f(\ell_\beta)
\bar\ell_\beta\in\pi_1(X,\alpha_0)$ represents $h(\beta)$
(equivalently, $\bar\rho_e f(e)\bar e\rho_e$ represents $h(\beta)$ for
a suitable ray $\rho_e$ in $T$ with $e$ oriented towards
$\beta$). Since $h(\alpha_0)=1$ the map $f$ will be identity on
$X_g^*$, as no edges $e$ where we modify the identity are along the
ray to $\alpha_0$.  We must show that $f$ is a proper map. Let
$\beta\in \partial X$ and assume that the edge $e$ as above is close
to $\beta$. This means that $\beta\in \partial X_g$ (ends in $DX$ have
a neighborhood not containing any edges $e$ as above). It follows that
$f(e)$ is close to $\beta$ by (R2). 
Thus $f$ is a proper map. Let $g$ be constructed
similarly for $h^{-1}\in \cR$. Then $gf\in PMaps(X)$ has the property
that $\Phi_T(gf)=\Phi_T(g)\Phi_T(f)=h^{-1}h=id$, so $gf\simeq id$ by
the injectivity of $\Phi_T$, and similarly $fg\simeq id$.
\end{proof}

We now have a useful criterion when a proper map $X\to X$ is properly
homotopic to the identity, without assuming it is a proper homotopy
equivalence. 

\begin{cor}\label{criterion}
  Let $f:X\to X$ be proper. Then $f$ is properly homotopic to the
  identity if and only if it preserves the homotopy class of every
  oriented closed curve and the proper homotopy class of every
  oriented proper line in
  $X$ that in each direction converges to an end in $DX=\partial
  X\smallsetminus \p X_g$.
\end{cor}

\begin{proof}
  If $X$ is a core graph we are assuming that $f_*\in Out(\pi_1(X))$
  preserves all conjugacy classes in $\pi_1(X)$. This implies that $f_*=id$ 
  and the conclusion follows from Theorem \ref{id}.
  If $DX\neq\emptyset$ choose some $\al_0\in DX$. We now see that
  $f_*:\pi_1(X,\al_0)\to\pi_1(X,\al_0)$ is identity, since $f$
  preserves lines that start and end at $\al_0$, and similarly $f$
  fixes all ends of $X$. Finally, we see that
  $\Phi_T(f)=1$ since $f$ preserves all lines joining $\al_0$ with any
  $\beta\in DX$, so the statement follows from Theorem \ref{kernel}.
\end{proof}

\begin{cor}\label{criterion2}
  Suppose $f:X\to Y$ is proper, induces a homeomorphism $\p X\to\p Y$
  and the restriction $X_g\to Y_g$ is a proper homotopy equivalence. Then $f$
  is a proper homotopy equivalence.
\end{cor}

\begin{proof}
  Using Proposition \ref{extension} we have a proper map $g:Y\to X$ so
  that the restriction $Y_g\to X_g$ is the homotopy inverse to $f:X_g\to Y_g$
  and so that $fg$ and $gf$ are identity on the
  boundaries. But then both are proper homotopy equivalences by
  Theorem \ref{kernel}, and thus both $f$ and $g$ are as well.
\end{proof}

Recall that $\cR \cong Ker(P\maps(X) \to P\maps(X_g^*))$.

\begin{cor}
  If $\alpha_0\in DX$ and
  if $\pi_1(X,\alpha_0)$ is nonabelian then
  $$K=Ker(P\maps(X)\to P\maps(X_g))$$ 
  fits in an exact sequence
  $$1\to \cR\to K\to \pi_1(X,\alpha_0)\to 1$$
  If $X_g=S^1$ then $Ker(P\maps(X)\to P\maps(X_g)=\Z/2\Z)$ is isomorphic to
  $\cR$. 
\end{cor}

\begin{proof}
  We focus on the first statement. The horizontal and vertical
  sequences in the commutative diagram below are exact by Theorem
  \ref{kernel} and Lemma \ref{Ystar}. The diagonal sequence is exact
  by the definition of $K$. The construction of the red arrows and the
  exactness of the sequence are a diagram chase.
\vskip 0.5cm

  \begin{center}
  \begin{tikzcd}
    1
    \arrow{dr} & & & 1\arrow{d} & \\
    & K \arrow{dr}
    \arrow[red]{rr} & & \pi_1(X,\alpha_0) \arrow{d}\arrow[red]{r} & 1\\
    1\arrow{r} & \cR \arrow{r}\arrow[red]{u} & P\maps(X) \arrow{r}\arrow{dr} &
    P\maps(X_g^*) \arrow{r}\arrow{d} & 1\\
    & 1\arrow[red]{u} & & P\maps(X_g) \arrow{d} \arrow{dr}\\
    & & & 1 & 1
    \end{tikzcd}
\end{center}
\end{proof}

When $DX$ is compact there is a more refined statement. Note that in
that case condition (R2) in the definition of $\cR$ is vacuous.

\begin{cor}\label{semidirect}
  If $DX$ is compact and nonempty then
  $$P\maps(X)\cong \cR\rtimes P\maps(X_g^*)$$
  where $P\maps(X_g^*)$ acts on $\cR$ by $g\cdot h(\beta)=g_*(h(\beta))$,
  where $g_*:\pi_1(X_g^*,\al_0)\to \pi_1(X_g^*,\al_0)$ is the homomorphism
  induced by $g$.
\end{cor}

\begin{proof}
  By Definition \ref{defOfPhi}, we have $\Phi_T:P\maps(X)\to\cR$, and let
  $R:P\maps(X)\to P\maps(X_g^*)$ be the restriction. Thus we have a
  function
  $$(\Phi_T\times R):P\maps(X)\to \cR\rtimes P\maps(X_g^*)$$
  That this is a homomorphism follows from the displayed calculation
  in the proof of Theorem \ref{kernel}. That the map is 1-1 and onto
  follows from Theorem \ref{kernel} plus the observation that $R$ is
  onto.
\end{proof}

When $DX$ is not compact, $\Phi_T$ may not be a well-defined function
to $\cR$ since (R2) may fail.

\section{Topology}\label{topology}

It
takes a bit of care to define the topology on $\maps(X)$.
Let $\hat X$ denote the Freudenthal compactification $X\cup \partial
X$ by the ends of $X$. If a map $f:X\to X$ is a proper homotopy
equivalence then it induces an isomorphism of $\pi_1(X)$ and it
extends to a continuous map $\hat X\to\hat X$ that restricts to a
homeomorphism of $\partial X$. However, the converse of this statement
is false.

\begin{example}\label{noninvertible}
  Let $X$ be the ray $[0,\infty)$ with the circle $x_n$ attached at
    $n$. Then there is a proper map whose action on $\pi_1$ is
    $x_0\mapsto x_0$ and $x_n\mapsto x_n x_{n-1}$ for $n>0$. The
    inverse sends every $x_n$ to a word that involves $x_0$ so it
    cannot be realized by a proper map. 
\end{example}

To circumvent this pathology, we will consider the space of pairs of
maps which are proper homotopy inverses of each other. This way the
inverse is ``built in'', cf. the proof that inversion is continuous,
Proposition \ref{group}. More precisely, let $(\hat X\to\hat X)$ be
the space of all continuous maps $\hat X\to \hat X$ equipped with
compact open topology. If we fix a metric $d$ on $\hat X$ then $(\hat
X\to\hat X)$ has the associated $\sup$ metric which we also denote by
$d$. This
metric is also complete, and composition is continuous, and $(\hat
X\to\hat X)$ is separable (see e.g. \cite[Theorem 4.19]{Kechris_sep}).

Next, we look at the space $PH(X)\subset (\hat X\to\hat X)^2$ consisting
of pairs $(\hat f,\hat g)$ such that:

\noindent
{\it $\hat f,\hat g$ are extensions to $\hat X$ of proper homotopy
equivalences $f,g:X\to X$ that are each other's inverses.}

In particular, $\hat f,\hat g$ are homeomorphisms when restricted to
$\partial X$ and they are each other's inverses. We put the product
topology on $(\hat X\to\hat X)^2$ and the subspace
topology on $PH(X)$.

Now define the function $\pi:PH(X)\to \maps(X)$ by
$$\pi(\hat f,\hat g)=[f]$$ where $[f]$ is the proper homotopy class of
the restriction of $\hat f$ to $X$. This function is surjective and
we put the quotient topology on $\maps(X)$.

\begin{prop}
  The quotient map $\pi$ is an open map.
\end{prop}

\begin{proof}
   Let $U\subseteq PH(X)$ be open; we need to show that
   $\pi^{-1}\pi(U)\subseteq PH(X)$ is open. Let $(\hat f,\hat g)\in
   \pi^{-1}\pi(U)$. Thus we have proper homotopies $H,K$ from $f,g$ to
   $h,k$ respectively, and $(\hat h,\hat k)\in U$. Therefore there is
   $\epsilon>0$ such that if $(\hat h',\hat k')\in PH(X)$ and $d(\hat
   h,\hat h')<\epsilon$, $d(\hat
   k,\hat k')<\epsilon$ then $(\hat h',\hat k')\in U$. We now claim
   that there is $\delta>0$ such that if $(\hat f',\hat g')\in PH(X)$,
   $d(\hat f,\hat f')<\delta$, $d(\hat g,\hat g')<\delta$, then there
   are homotopies of $f',g'$ to maps $h',k'$ as above, and this will
   show that $(\hat f',\hat g')\in\pi^{-1}\pi(U)$, finishing the
   proof.

   We will prove the claim using Proposition \ref{pasting_homotopies}. Note
   that all proper homotopies between maps on $X$ extend continuously
   to $\partial X$ and are stationary on all points of $\partial
   X$. It follows that in the complement of a sufficiently large
   finite subgraph the tracks (i.e. paths traversed by points) of each
   such 
   homotopy are as small as we
   like. In addition, the edges outside a large finite subgraph are as
   small as we like.
   
   Choose a large finite subgraph $L\subset X$ and
   choose $\delta_f>0$ so that if $d(\hat f,\hat f')<\delta_f$ then we
   have a homotopy between $f'|L$ and $f|L$ whose tracks have small
   size, and we also have a homotopy between $f$ and $h$ with small
   tracks for points in $\overline{X\smallsetminus L}$. Applying Proposition
   \ref{pasting_homotopies} we get a homotopy from $f'$ to some map
   $h'$ that agrees with $h$ on $L$ and whose tracks of points outside
   $L$ are small. Thus $h'$ is close to $f$ outside $L$, which in turn
   is close to $h$ outside $L$. Thus $h'$ is close to $h$
   everywhere. In a similar way we find $\delta_g>0$ and a homotopy
   from $g'$ to $k'$. Then we set $\delta=\min\{\delta_f,\delta_g\}$.
   \end{proof}

\begin{cor}\label{phe}
  Let $PHE(X)\subset (\hat X\to\hat X)$ be the subspace of maps $\hat
  f$ that are extensions to $\hat X$ of proper homotopy equivalences
  $f \from X\to X$. Then the map $q\from PHE(X)\to \maps(X)$ that to $\hat f$ assigns
  $[f]$ is open. Thus alternatively we could use this map to define
  the quotient topology on $\maps(X)$.
\end{cor}

\begin{proof}
  The projection $(\hat X\to\hat X)^2\to(\hat X\to\hat X)$ to the
  first coordinate restricts
  to a continuous map $PH(X)\to PHE(X)$. The composition with
  $q\from PHE(X)\to \maps(X)$ is open by the proposition, so $q$ is open.
\end{proof}

\begin{prop}\label{group}
  With this topology $\maps(X)$ is $T_1$ and it is a topological group.
\end{prop}

\begin{proof}
  We first show $\maps(X)$ is a topological group. That composition is
  continuous follows from the fact that the map $PH(X)\times PH(X)\to
  PH(X)$ defined by $((\hat f,\hat g),(\hat f',\hat g'))\mapsto (\hat
  f\hat f',\hat g'\hat g)$ is continuous, being the restriction of the
  analogous map $(\hat X\to\hat X)^4\to (\hat X\to\hat X)^2$.
  \vskip 1cm
\begin{center}
  \begin{tikzcd}
    PH(X)\times PH(X) \arrow{d}{\pi\times\pi} \arrow[r] &
    PH(X)\arrow{d}{\pi} \\
    \maps(X)\times \maps(X)\arrow[r] & \maps(X)
  \end{tikzcd}
  \vskip 1cm
\end{center}
  Product of open maps is open so both vertical arrows are quotient
  maps. Continuity of the bottom horizontal map now follows.

  For the inverse we consider the map $PH(X)\to PH(X)$, $(\hat f,\hat
  g)\mapsto (\hat g,\hat f)$ and the argument is similar.

  To prove that $\maps(X)$ is $T_1$ it suffices to show that the
  identity point is closed, since $\maps(X)$ is a topological
  group. This amounts to showing that the space of $\hat f:\hat
  X\to\hat X$ so that $f$ is properly homotopic to the identity is a
  closed set. This follows from Corollary \ref{criterion} since
  preserving a loop or a line is a closed condition.
\end{proof}

\begin{cor}\label{topiso}
  Let $X,Y$ be two connected locally finite graphs that are proper
  homotopy equivalent. Then $\maps(X)$ and $\maps(Y)$ are isomorphic as
  topological groups.
\end{cor}

\begin{proof}
  Let $F:X\to Y$ and $G:Y\to X$ be proper homotopy equivalences with
  $FG$ and $GF$ properly homotopic to the identity. Then $F,G$ extend
  to maps $\hat F:\hat X\to \hat Y$ and $\hat G:\hat Y\to \hat X$ and
  we have the induced maps $(\hat X\to\hat X)\to (\hat Y\to \hat Y)$,
  $\hat f\mapsto \hat F\circ \hat f \circ \hat G$ and $(\hat Y\to \hat Y)\to (\hat
  X\to\hat X)$, $\hat g\mapsto \hat G \circ \hat g \circ \hat F$. These maps
  restrict to $PHE(X)\to PHE(Y)$ and $PHE(Y)\to PHE(X)$,
  which induce homomorphisms $\maps(X)\to \maps(Y)$ and $\maps(Y)\to
  \maps(X)$. These are each other's inverses and they are both
  continuous by Corollary \ref{phe}.
\end{proof}

\subsection{Clopen subgroups}\label{clopen}
We next show that $\maps(X)$ has a countable basis consisting of clopen
sets, which are in fact cosets of subgroups. We will use this to show
that $\maps(X)$ is homeomorphic to $\Z^\infty:=\prod_1^\infty \Z$, i.e. the irrationals,
and is hence a totally disconnected Polish group.

We define clopen subgroups of $\maps(X)$, analogs of pointwise stabilizers of
compact subsurfaces in big mapping class groups. Let $K\subset X$ be a
finite subgraph and define
$U_K$ to be the set of equivalence classes $[f]\in \maps(X)$ with a representative such that:
\begin{enumerate}[(i)]
\item $f=id$ on $K$,
\item $f$ preserves each complementary component of $K$,
  \item there is a representative $g$ of $[f]^{-1}$ that also
    satisfies (i) and (ii),
    \item there are proper homotopies $gf\simeq 1$ and $fg\simeq 1$
      that are stationary on $K$ and preserve complementary components
      of $K$.
\end{enumerate}

\begin{lemma}\label{UkOpen} $U_K$ is open.
\end{lemma}
\begin{proof}
We must show that $\pi^{-1}(U_K)$ is open.  Let $\pi(h,k) \in U_K$.
Let $f,g$ be the representatives satisfying (i)-(iv).  We then have
proper homotopies $H,H':X\times I\to X$ with $H(x,0)=h(x), H(x,1)=f(x),
H'(x,0)=k(x), H'(x,1)=g(x)$. Choose a compact subgraph $L\supset K$ such
that both $H$ and $H'$ at all times map $X\smallsetminus L$ into $X
\smallsetminus K$. Finally, choose $\epsilon>0$ such that if
$d(h',h)<\epsilon$ and $d(k',k)<\epsilon$ then $h',k'$ send
$X\smallsetminus L$ to $X\smallsetminus K$ and $h,h'$ ($k,k'$)
restricted to $L$ are homotopic by a homotopy that doesn't move points
in $Fr(L)$ into $K$.

Now use Proposition \ref{pasting_homotopies} a total of 4 times to prove
that $\pi(h',k')\in U_K$. First, we have a homotopy from $h'$ to a map $h''$
extending $h|L$, and second, we have a homotopy from $h''$ to a
map $f'$ extending $f|L$. Thus $f'|K=1$ and preparations above show
that $f'$ maps $X\smallsetminus L$ to $X\smallsetminus K$. Two more
homotopy extensions yield a similar homotopy from $k'$ to $g'$. Since
$(f',g')$ satisfies (i)-(iv) we conclude that $\pi(h',k')\in U_K$.
\end{proof}

\begin{prop}
  $U_K$ is a clopen subgroup. For every neighborhood $U$ of $1\in
  \maps(X)$ there is some $K$ so that $U_K\subseteq U$.
\end{prop}

\begin{proof}
To prove that $U_K$ is a subgroup, we must prove that if $[f_1],[f_2]
\in U_K$ with $f_1,f_2$ preferred representatives, then $[f_1][f_2]
\in U_K$. Indeed, $[f_1][f_2] = [f_1 \circ f_2]$ and $f_1 \circ f_2$ is
the identity on $K$ and preserves complementary components. The same
is true for its homotopy inverse $g_2 \circ g_1$ as well as for both
homotopies to $id_K$.

We proved in Lemma \ref{UkOpen} that $U_K$ is open. Thus $U_K$ is also
closed since its complement is a union of cosets, which are also open.

Finally, we must show that for every open $U \subset \maps(X)$ containing $1$, there is a compact $K$ such that $U_K \subset U$.
We have $\pi^{-1}U$ is an open set in $PH(X)$ containing $(1,1)$.
Let $\epsilon>0$ be such that if 
$(f, g)\in PH(X)$ and $d(1,\hat f)<\epsilon$, $d(1,\hat g)<\epsilon$
then $(\hat f,\hat g)\in \pi^{-1}(U)$. 
Let $K\subset X$ be compact so
that all complementary components of $K$ have diameter
$<\epsilon$. Then for preferred representatives $(f,g)$ of $[f] \in U_K$ we have
$(f, g)\in\pi^{-1}(U)$, so $U_K\subseteq \pi\pi^{-1}(U)=U$. 
\end{proof}

\begin{cor}
  $\maps(X)$ has a countable basis of clopen sets, it is separable,
  metrizable and totally disconnected.
\end{cor}

\begin{proof}
  $\maps(X)$ is separable since it is the continuous image of $PH(X)$
  which is separable, being a subspace of a separable metric space. So
  in particular each open subgroup $U_K$ has at most countably many
  cosets.  Choose an exhaustion $K_i$ of $X$; then $U_{K_i}$ and their
  cosets form a countable basis of clopen sets.  Next, $T_1$ plus a
  basis of clopen sets implies regular (proof: let $x\notin A$ with
  $A$ closed; then there is a basis element $V$ with $x\in V$ and
  $V\cap A=\emptyset$ and so $V,V^c$ is the required separation). Then
  countable basis plus regular implies normal
  \cite[1.5.16]{engelking}, and finally countable basis plus normal
  implies metrizable (Urysohn metrization theorem, \cite[Theorem
    34.1]{munkres}).
\end{proof}

Recall that $X$ has {\it finite type} if $\pi_1(X)$ is finitely
generated and $\partial X$ is finite; otherwise $X$ has {\it infinite
  type}. 

\begin{lemma}\label{not compact}
  Suppose that $X$ has infinite
  type. Then for every finite subgraph $K\subset X$ there is a finite
  subgraph $L\subset X$ such that $L\supset K$ and $U_L<U_K$ has
  infinite (countable) index.
\end{lemma}

\begin{proof}
  First recall that an infinite totally disconnected compact
  metrizable space either has infinitely many isolated points or else
  it is the disjoint union of a Cantor set and finitely many isolated
  points. Consider the complementary components of $K$. If one of them
  has genus $N>1$ then there is $L\supset K$ such that $U_K$ contains
  a subgroup $H$ isomorphic to the infinite group $Aut(F_N)$, while
  $H\cap U_L=1$. The elements of $H$ are realized by homotopy
  equivalences supported on a finite subgraph of genus $N$. 
  Otherwise, after
  perhaps enlarging $K$, we may assume that all complementary
  components are trees, and in this case $\partial X$ is
  infinite. According to the dichotomy above, we can find $L\supset K$
  such that one of the following holds:
  \begin{itemize}
    \item There is a complementary component of $K$ whose boundary at infinity
      contains an infinite set of isolated points, and this set can be
      written nontrivially as $A\cup B$ where $A$ and $B$ are boundary points of
      two distinct complementary components of $L$.
      \item There is a complementary component of $K$ whose boundary
        at infinity contains a Cantor set as a clopen subset, and this
        Cantor set can be written nontrivially as $A\cup B$ with both
        $A,B$ clopen and both boundary points of two distinct
        complementary components of $L$.
  \end{itemize}

  In the first case, $U_K$ contains the entire group $Perm_0(A\cup B)$
  of finitely supported permutations of $A\cup B$ (cf. the
  classification theorem), while the intersection of this group with
  $U_L$ contains only permutations that preserve $A$ and $B$. Since
  either $A$ or $B$ is infinite, this subgroup has infinite index.

  In the second case $U_K$ contains the entire group $H=Homeo(A\cup B)$,
  while the intersection $H\cap U_L$ contains only homeomorphisms that
  preserve $A$ and $B$, and this again has infinite index.
\end{proof}

\begin{lemma}
  $PH(X)\subset (\hat X\to\hat X)^2$ is a $G_\delta$-subset.
\end{lemma}

\begin{proof}
  A pair $(\hat f,\hat g)\in (\hat X\to\hat X)^2$ is in $PH(X)$ iff
  (see Corollary \ref{criterion}):
  \begin{enumerate}[(1)]
    \item $\hat f(\partial X)\subseteq \partial X$, $\hat g(\partial
      X)\subseteq \partial X$,
      \item $\hat f(X)\subseteq X$, $\hat g(X)\subseteq X$,
      \item $\hat f\hat g$ and $\hat g\hat f$ are identity on $\partial X$,
        \item $\hat f\hat g$ and $\hat g\hat f$ restricted to $X$
          preserve the homotopy classes of oriented loops and proper
          homotopy classes of oriented proper lines joining points of $DX$. 
  \end{enumerate}

Conditions (1),(3) and (4) are closed conditions and (2) is
$G_\delta$: (2) can be written as countably many conditions $\hat
f(K_n)\subset X$, $\hat g(K_n)\subset X$ for an exhaustion $\{K_n\}$,
$n=1,2,\cdots$ and these are all open.
\end{proof}

When $X$ has finite type, $\maps(X)$ is countable and
discrete. Otherwise we have:

\begin{prop}
  Suppose $X$ has infinite type. Then $\maps(X)$ is a Polish group
  with the underlying space
  homeomorphic to $\Z^\infty$, i.e. to the set of irrationals.
\end{prop}

\begin{proof}
There is a theorem of Sierpinski that if $f\from X\to Y$ is an open
surjective map between separable metric spaces and $X$ is complete,
then $Y$ is completely metrizable (see \cite[Exercise
  5.5.8.(d)]{engelking}).  Since $PH(X)$ is a $G_\delta$ subset of
$(\hat X\to\hat X)^2$, it is completely metrizable (see
\cite[4.3.23]{engelking}) and therefore $\maps(X)$ is completely
metrizable.  Then the fact that $\maps(X)$ is homeomorphic to the
irrationals follows from a theorem of Hausdorff (see
e.g. \cite{eberhart}): If $Z$ is separable, completely metrizable,
zero dimensional (i.e. has a basis of clopen sets), and every compact
subset has empty interior, then $Z$ is homeomorphic to the
irrationals.  To finish the proof, if $\maps(X)$ had a compact subset
with nonempty interior, then some $U_{K}$ would be compact. But this
contradicts Lemma \ref{not compact}, since $U_{K}$ is covered by the
pairwise disjoint cosets of $U_{L}$ and this cover doesn't have a
finite subcover.
\end{proof}

We finish this section by considering continuity properties of
homomorphisms studied in Section \ref{s:algebra}.

Recall the surjective homomorphism $\sigma \from \maps(X)\to Homeo(\p X,\p
X_g)$ from Corollary \ref{sigma}. The group $Homeo(\p X,\p X_g)$ is
equipped with the compact-open topology. This means that a basis of
neighborhoods of the identity is defined by clopen subgroups
$V_{\mathcal P}$ where $\mathcal P$ is a finite partition of $\p X$
into clopen subsets and $V_{\mathcal P}$ consists of the elements of
$Homeo(\p X,\p X_g)$ that leave the partition elements
invariant. Refining the partition yields a smaller clopen subgroup.

\begin{cor}\label{4.12}
  The homomorphism $\sigma$ is continuous and open. In particular, when $X$
  is a tree, $\sigma:\maps(X)\to Homeo(\p X)$ is an isomorphism of
  topological groups.
\end{cor}

\begin{proof}
  We may assume that $X$ is a Standard Model. We will consider finite
  subgraphs $K\subset X$ consisting of a subtree in the
  underlying tree together with all circles attached to it. The
  complementary components of $K$ determine a partition $\mathcal P_K$
  of $\p X$. Since every partition $\mathcal P$ is refined by some
  $\mathcal P_K$ and $\sigma(U_K)\subseteq V_{\mathcal P_K}$ it
  follows that $\sigma$ is continuous. To prove that $\sigma$ is open
  it suffices to argue that $\sigma(U_K)=V_{\mathcal P_K}$. Let $W$ be
  a complementary component of $K$. Thus $\p W$ is one of the
  partition elements $A_W$ of $\mathcal P_K$ together with one point $v$
  corresponding to the vertex of intersection
  $\overline W \cap K$. Given a homeomorphism $h$ of $(A_W,A_W\cap\p
  X_g)$, extend it by $v\mapsto v$ and view it as a homeomorphism of
  $(\p W,\p W_g)$. By the Classification Theorem there is
  $f_h\in\maps(W)$ that induces $h:\p W\to\p W$. Now define
  $f\in\maps(X)$ as the identity on $K$ and as $f_h$ on $W$, for each
  complementary component $W$, and observe that $\sigma(f)$ is the
  given homeomorphism in $V_{\mathcal P_K}$.
\end{proof}

Next, recall the homomorphism $\Psi:Maps(X)\to Out(\pi_1(X))$ to the
Polish group $Out(\pi_1(X))$. It is injective when $X$ is a core graph
(Theorem \ref{id}). 

\begin{prop}
  The homomorphism $\Psi$ is continuous. If
  the genus of $X$ is infinite, the image is not a closed subgroup. If
  in addition $X$ is a core graph then $\Psi$ is injective but it is
  not a homeomorphism onto its image.
\end{prop}

\begin{proof}
  The topology on $Aut(\pi_1(X))$ is defined as a subgroup of the
  symmetric group $S_\infty$ on the countable set $\pi_1(X)$, so an
  automorphism is close to the identity if it fixes a large finite
  set. The group
  $Out(\pi_1(X))$ is equipped with the quotient topology. If $U$ is an
  open neighborhood of the identity in $Out(\pi_1(X))$, its preimage
  in $Aut(\pi_1(X))$ will contain all automorphisms that fix a certain
  finite set $F$. The elements of $F$ are realized inside some compact
  subgraph $K\subset X$ and it follows that $\Psi(U_K)\subseteq U$, so
  $\Psi$ is continuous.

  Consider $f \from X\to X$ from Example \ref{noninvertible}. Let $f_n \from X\to
  X$ be defined by $f_n(x_0)=x_0$, $f_n(x_k)=x_kx_{k-1}$ for $k\leq n$
  and $f_n(x_k)=x_k$ for $k>n$. Then $\Phi(f_n)\to f_*\in
  Out(\pi_1(X))$, but $f_*$ is not in the image of $\Phi$.

  Similarly, consider $g_n \from X\to X$ defined by $g_n(x_k)=x_k$
  when $k\leq n$ or $k\geq 2n$, $g_n(x_k)=x_kx_{1}$ when $n<k<2n$. Then
  $\Psi(g_n)\to id$ but the sequence $g_n$ does not converge to $id$
  (or anywhere). So $\Psi$ is not a homeomorphism onto its image.

  Generalizing these examples to other graphs is left to the reader.
\end{proof}

Finally, we have the following statement, whose proof is left to the
reader.

\begin{prop}
  The restriction epimorphisms $PMaps(X)\to PMaps(X_g)$ and $PMaps(X)\to
  PMaps(X_g^*)$ are continuous and open.
\end{prop}

\section{Proof of Main Theorem for core graphs}

\subsection{Free factor systems}
\def\F{{\mathbb F}} Let $\F$ be a free group, possibly of infinite
rank. Recall that a nontrivial subgroup $A<\F$ is a {\it free factor}
of $\F$ if there is a subgroup $B<\F$ such that $A*B=\F$. We will only
consider free factors of finite rank, and only conjugacy classes $[A]$
of such free factors. To simplify notation we will usually omit the
brackets. Topologically, a (conjugacy class of a) nontrivial subgroup
is a free factor if there is a graph $\Gamma$ with $\pi_1(\Gamma)=\F$
and with $A$ represented by a subgraph. Similarly, a finite collection
$\mathcal F$ of (conjugacy classes of) finitely generated free factors
is a {\it free factor system} if there are representatives
$A_1,A_2,\cdots,A_n$ and a subgroup $B<\F$ such that $A_1*A_2*\cdots
*A_n*B=\F$. Topologically, there is a graph $\Gamma$ with
$\pi_1(\Gamma)=\F$ and with the $A_i$s represented by pairwise
disjoint subgraphs.

If $\mathcal F$ and $\mathcal F'$ are two free factor systems, the
intersection $\mathcal F\cap\mathcal F'$ is naturally a free factor
system. It consists of conjugacy classes of nontrivial subgroups
obtained by intersecting a representative of a conjugacy class in
$\mathcal F$ with a representative of a conjugacy class in $\mathcal
F'$. Topologically, one can represent $\mathcal F$ and $\mathcal F'$
by immersions of finite graphs $\Gamma_F\to\Gamma$ and
$\Gamma_{F'}\to\Gamma$, form the pull-back (see \cite{stallings}) and
discard the contractible components to get an immersion representing
the intersection.

\begin{example}
  Let $\F=\<a,b,c\>$, $A=\<a,b\>$, $B=\<a,cbc^{-1}\>$. Then $A$ and
  $B$ are free factors of $\F$, while their intersection is the free
  factor system consisting of two rank 1 free factors $\<a\>$ and
  $\<b\>$. The intersection of $A$ and $\<c\>$ is the empty free
  factor system. 
\end{example}

To see that the intersection $\mathcal F\cap\mathcal F'$ is a free
factor system, one can arrange that one of them is represented by
subgraphs of $\Gamma$ and then the pullback will be represented by
subgraphs of the other one. It is also possible to compute finite
intersections of free factor systems by a pull-back of several
immersions.

Finally, we write $\mathcal F<\mathcal F'$ if every group (representing 
a conjugacy class) in $\mathcal F$ is contained in a group in $\mathcal
F'$. For example, $\mathcal F\cap\mathcal F'<\mathcal F$.

\subsection{Tree of groups}
\def\F{{\mathcal F}} We now assume that $X$ is a core graph and is a
Standard Model. Thus $X$ is a tree $T$ with a root vertex $v$ and with
a loop attached at every vertex. We assign length 1 to each edge and
let $D_0:T\to [0,\infty)$ be the distance function from $v$. We extend
  $D_0$ to all of $X$ so that it is constant on each attached
  loop. Our first task is to control the sizes of maps, measured in
  $[0,\infty)$, representing
  elements of $H$, as well as homotopies, measured by $D_0$. So in
  effect we replace properness by metric control. Recall that for a
  finite subgraph $K\subset X$ we have a clopen subgroup $U_K<Maps(X)$,
  so $H\cap U_K$ is compact and has finite index in $H$.

 \begin{prop}\label{rs}
    Let $H<\maps(X)$ be a compact subgroup. There is a sequence of
    integers $0=r_0<r_1<r_2<\cdots$ and for every $n>0$ and every
    $[h]\in H$ there is a representative $h$ satisfying
\begin{enumerate}[(*)]
      \item $h$ maps
        every element of the closed cover $\mathcal
        C(r_1,r_2,\cdots,r_n)$ of $X$ consisting of the sets
        $$D_0^{-1}[r_0,r_1],D_0^{-1}[r_1,r_2],\cdots,D_0^{-1}[r_{n-1},r_n],D_0^{-1}[r_n,\infty)$$
          to the union of the same element with the one or two
          adjacent elements.
\end{enumerate}
\end{prop}

  \begin{proof}
    We construct the numbers inductively, starting
    with $r_1=1$. Then (*) is
    vacuous. 

    Suppose that $r_n$ has been constructed satisfying (*). Note that
    by properness for every $[h]\in H$ (and every representative $h$
    that exists by induction) there is some $r_{n+1}>r_n$ so that
    $(*)$ holds for the cover $C(r_1,r_2,\cdots,r_n,r_{n+1})$ and this
    $h$. Moreover, the same $r_{n+1}$ will also work in a neighborhood
    of $[h]$ by choosing representatives of the form $hu$ where $[u]\in
    U_{D_0^{-1}[0,r_{n+1}]}$ i.e. $u$ fixes $D_0^{-1}[0,r_{n+1}]$ and
    leaves the complementary components invariant. Now by compactness
    of $H$, there is a finite cover of $H$ by such open sets and the
    maximal $r_{n+1}$ will then satisfy the requirements.
  \end{proof}

  It will be convenient to introduce the following notation. First,
  let $\rho:[0,\infty)\to [0,\infty)$ be a homeomorphism such that
      $\rho(r_n)=n$ for $n=0,1,\cdots$ and let $D=\rho D_0:X\to [0,\infty)$. Thus
      $D^{-1}([m,n])=D_0^{-1}([r_m,r_n])$. We think of $D$ as a
        ``control function''. For example, (*) says that for every $n$
        every element
        of $H$ has a representative $h$ that ``moves points $<2$'' 
        i.e. $|D(x)-D(h(x))|<2$ for every $x\in D^{-1}[0,n]$. The next
        proposition says that 
        homotopies ``move
        points $<3$''.

  \begin{prop}\label{homotopies}
    Let $0<r_1<r_2<\cdots$ be as in Proposition
    \ref{rs}. Fix $n$ and let $h,h'$ be the representatives of two
    elements of $H$ that are inverses of each other as in Proposition
    \ref{rs}. Then
    there is a proper homotopy between the identity and $h'h$ that
    moves each element of the cover $\mathcal C(r_1,\cdots,r_n)$ to the union of
    at most 5 elements, namely the 2-neighborhood of the given element.
  \end{prop}

  \begin{proof}
    First note that there is a canonical proper homotopy between the
    identity and $h'h$: lift the given homotopy to the universal cover
    extending the identity map, and then replace it by the straight
    line homotopy. We now argue that this homotopy moves within
    2-neighborhoods. Fix a component $P$ of an element of the cover
    and let $\tilde P$ be the component of the 2-neighborhood that
    contains it. Since a loop in $P$ cannot be mapped by $h'h$ disjointly
    (since otherwise $h'h$ would not be homotopic to the identity) we
    see that $h'h(P)\subseteq \tilde P$. By lifting to the covering
    space of $X$ corresponding to $\pi_1(\tilde P)$ and then
    retracting to the core $\tilde P$, we see that
    $h'h|P:P\to \tilde P$ is homotopic to inclusion
    $i:P\hookrightarrow\tilde P$ within $\tilde P$. Now note that any
    homotopy from $i$ to $h'h|P$ has tracks that are nullhomotopic loops
    (they have to represent $\pi_1$-elements that commute with
    $\pi_1(P)$, but since $\pi_1(P)$ and $\pi_1(\tilde P)$ are
    nonabelian free groups this forces these loops to be trivial). It
    follows that the tracks described by the straight line homotopy
    are homotopic to paths in $\tilde P$, but since they are immersed,
    they must be contained in $\tilde P$.
  \end{proof}

  If $J\subset [0,\infty)$ is a closed interval with integer
    endpoints, write $\mathcal F(J)$ for the free factor system
    represented by $D^{-1}(J)$. Thus the number of free factors in
    $\mathcal F(J)$ is equal to the number of components of
    $D^{-1}(J)$. When $J$ is a degenerate interval (a single integer
    point) then each factor in $\F(J)$ has rank 1. We denote by $|J|$
    the length of the interval.

  We also set $$\F'(J)=\cap_{h\in H}h_*(\F(J))$$
  where $h_*:\pi_1(X)\to\pi_1(X)$ is the automorphism induced by $h$
  (defined up to conjugation). This is really only a finite
  intersection since when $h$ is close to the identity we will have
  $h_*(\F(J))=\F(J)$, so it suffices to intersect over finitely many
  coset representatives. Thus $\F'(J)$ is an $H$-invariant free factor
  system.

When $J=[a,b]\subset [0,\infty)$ with integer endpoints and with
  $b-a\geq 4$ we set $J^+=[a-2,b+2]\cap [0,\infty)$ to be the
    2-neighborhood of $J$, and likewise $J^-=[a',b-2]$ where $a'=0$ if
    $a=0$ and otherwise $a'=a+2$ (so $J^-$ is obtained from $J$ by
    subtracting the 2-neighborhood of the complement).
    Note
    that by our assumptions on the sequence $r_n$ we have that
  $$\F(J^-)<\F'(J)<\F(J^+)$$
  We now show that each
  group in $\F'(J)$ either contains a group in $\F(J^-)$ or it has
  trivial intersection with all of them.

  \begin{lemma}\label{5.2}
    Let $A$ be a free factor in $\F'(J)$. If $A$ contains a nontrivial
    element $\alpha$ that also belongs to a free factor $B$ in $\F(J^-)$ then
    $B<A$ (up to conjugacy).
  \end{lemma}

  \begin{proof}
    Represent different $h_*(\F(J))$, $h\in H$, by immersions of
    finite (possibly disconnected) graphs into $X$. The non-tree
    components of the pull-back then represent the free factors in
    $\F'(J)$. Since free factors (and free factor systems) are
    malnormal, if an immersion to $X$ lifts to the pull-back, it does
    so uniquely. Since an immersion representing $B$ lifts, it must
    lift to the component representing $A$, since this is where
    $\alpha$ lifts.
  \end{proof}

  We now set $\F^*(J)$ to be the free factor system consisting of the
  free factors in $\F'(J)$ that contain a free factor in
  $\F(J^-)$. Thus we still have
  $$\F(J^-)<\F^*(J)<\F(J^+)$$
  and also $J\subset J'$ implies $\F^*(J)<\F^*(J')$.

  \begin{lemma}
    If $|J|\geq 8$ then $\F^*(J)$ is $H$-invariant. 
  \end{lemma}

  \begin{proof}
    Take a free factor $A$ in $\F^*(J)$. It will
    contain an element $\alpha$ corresponding to a loop in $D^{-1}(t)$
    for any $t\in J$ whose distance to each endpoint is $\geq 4$. For
    any $h\in H$ we have $h_*(\alpha)$ is an element in a free factor
    of $\F([t-2,t+2])<\F(J^-)$, and the free factor of the latter that
    contains it is contained in a free factor $B$ of $\F^*(J)$ by
    Lemma \ref{5.2}, and
    $h_*(A)=B$. 
  \end{proof}

Now fix a sequence of intervals $J_1,J_2,\cdots$ that cover
$[0,\infty)$ and so that $J_n\cap J_m=\emptyset$ when $|n-m|>1$ and
  $J_{n,n+1}:=J_n\cap J_{n+1}$ is an interval of length $\geq 22$
  for
  $i=1,2,\cdots$. Now construct the following tree of groups $\mathcal
  T$. The vertices of the tree are the free factors in $\F(J_n)$ (or
  equivalently the components of $D^{-1}(J_n)$), $n=1,2,\cdots$. The
  group associated to a vertex is the underlying free factor. The
  edges are the free factors in $\F(J_{n,n+1})$ (components of
  $D^{-1}(J_{n,n+1})$), again with the associated group the
  underlying free factor. Incidence relation is inclusion. The
  underlying graph is a tree, the nerve of the cover of $X$ by the
  components of $D^{-1}(J_n)$, $n\geq 1$.

  \def\T{{\mathcal T}}
  \begin{lemma}
    $\pi_1(\mathcal T)\cong \pi_1(X)$.
  \end{lemma}

  \begin{proof}
    By induction, the subtree of groups corresponding to the
    first $n$ intervals has the fundamental group of the corresponding
    subgraph of $X$.
  \end{proof}

  In a similar way we construct a tree of groups $\T^*$, which will be
  $H$-invariant. A vertex of {\it height $n$} is a free factor 
  in $\F^*(J_n)$, with this factor as the vertex group. An edge of {\it
    height $[n,n+1]$} is a free
  factor in $\F^*(J_{n,n+1})$, with this factor as the edge
  group. Such a factor is contained in a unique vertex group at height
  $n$ and a unique vertex group at height $n+1$ by Lemma \ref{5.2} and
  this gives incidence and edge-to-vertex inclusions. Thus $\T^*$ is a
  graph of groups and it is $H$-invariant by construction. Below we
  will show that $\T^*$ is a tree and $\pi_1(\T^*)\cong \pi_1(\T)$.

  \begin{lemma} 
    If $C$ is an edge group in $\T^*$ with $A,B$ the incident vertex
    groups of heights $n,n+1$ resp., then $C$ is one of the free
    factors in $A\cap B$. 
  \end{lemma}

  \begin{proof}
    We have that $C$ is contained in some group in $A\cap B$ by
    construction. The free factor $A$ is a free factor in the free
    factor system $\F'(J_n)$ that contains a factor in $\F(J_n^-)$ and
    similarly for $B$. The intersection $A\cap B$ consists of free factors in
    $\F'(J_{n,n+1})$ and one of them contains $C$, which is also a free factor
    in $\F'(J_{n,n+1})$, so equality holds.
  \end{proof}

  Note here that in principle the intersection of $A$ and $B$ can
  consist of several free factors, i.e. the vertices might be
  joined by several edges. We will rule this out in Lemma \ref{5.8}.

  There is a natural morphism
  (vertices to vertices and edges to edges) $\pi:\T^*\to\T$ that sends a
  factor in $\F^*(J_i)$ to the factor in $\F(J_i)$ that contains it,
  and similarly for the edges. Note that we have a height function on
  both trees (sending factors in $\T^*(J_i)$, respectively in
  $\F(J_i)$ to $i$) that commutes with this map.

  In the sequel it will be convenient to abuse the terminology and
  conflate a subcomplex of $X$ and its fundamental group, and likewise a
  ``component'' and a ``free factor'' in a free factor system. 

  \begin{lemma} 
    Every vertex of $\T^*$ at height $n+1>0$ is connected by an edge to a vertex
    at height $n$. There is a unique vertex of $\T^*$ at height
    0. In particular, $\T^*$ is connected. 
  \end{lemma}

  \begin{proof}
    Suppose the vertex is $A$, so it contains (possibly more than one)
    free factor $B$ in $\F(J^{-}_{n+1})$. Every component of
    $D^{-1}(J^-_{n+1})$ contains a (unique) component of
    $D^{-1}(J_{n,n+1}^-)$ and this component is contained in a unique
    free factor of $\F^*(J_{n,n+1})$, which represents an edge at
    height $[n,n+1]$ attached to $A$.

    Since $D^{-1}(J_1^{-})$ is connected (recall that $J_1^-$ contains
    $\{0\}$) and every vertex at height 0
    must contain a component of it, it follows that there is only one
    height 0 vertex in $\T^*$.
  \end{proof}

  Note that a vertex at height $n$ may not be connected to any
  vertices at height $n+1$ since a component of $D^{-1}(J^-_{n})$ may
  not contain any components of
  $D^{-1}(J_{n,n+1}^-)$.

  \begin{lemma}\label{5.8}
    Let $e$ be an edge in $\T$ with height in $[n,n+1]$ and consider
    its preimage $\pi^{-1}(e)$ in $\T^*$. After removing isolated
    vertices from $\pi^{-1}(e)$, it is a tree with one vertex $w$ at
    height $n$ and all other vertices at height $n+1$, and these are
    all connected to $w$ by a unique edge. In particular, $\T^*$ is a
    tree.
  \end{lemma}

  \begin{proof}
    Let $J=J_{n,n+1}$. The statement that all edges in the preimage of
    $e$ have the same vertex at height $n$ follows from the following
    fact. If two components of $D^{-1}(J^-)$ are contained in the same
    component of $D^{-1}(J)$ then they are contained in the same
    component of $D^{-1}(J_n^{-})$ (and this is not true if $J_n^-$ is
    replaced by $J_{n+1}^-$ and there may be several vertices at
    height $n+1$).

    We now argue that the height $n+1$ vertices of all these edges in
    the preimage of $e$ are
    distinct. Fix some integer $k\in J$ at
    distance $\geq 9$ from the endpoints and let $x,x'$ be two loops
    in $D^{-1}(J)$ that map to $k$. They will lift to unique components of
    $\F^*(J)$ and any two components are determined in this
    way. If they lift to the same component of $\F^*(J_{n+1})$ then
    there is an immersion $q:\Gamma\to D^{-1}(J_{n+1})$ of a barbell (two
    disjoint loops connected by an edge) sending one loop to $x$ and
    the other to $x'$ and so that $hq$ can be homotoped into
    $D^{-1}(J_{n+1})$ for every $h\in H$. Thus $q$ is kind of a
    ``witness'' that $x,x'$ lift to the same component of
    $\F^*(J_{n+1})$. We need a similar witness that they lift to the
    same component of $\F^*(J)$. The map $q$ may not work, since its
    image may contain points of $D^{-1}(J_{n+1}\smallsetminus J)$, and we will
    perform a kind of surgery on $q$ to get a better map. 

    Fix $h\in H$. By perturbing if necessary we may assume that $hq$
    doesn't collapse any edges and is simplicial with respect to
    suitable subdivisions. Then the statement that $hq$ can be
    homotoped into $\F(J_{n+1})$ is equivalent to saying that after
    folding and replacing $hq$ by an immersion, the image of the core
    subgraph is contained in $D^{-1}(J_{n+1})$. This same $q$ may not
    map into $D^{-1}(J)$ since it may map around loops in
    $D^{-1}(J_{n+1})\smallsetminus D^{-1}(J)$, so we will modify it to
    $q':\Gamma'\to D^{-1}(J)$. First
    we analyze $q$.

    Recall that a {\it
      vanishing path} for $hq$ is an immersion $\nu:I\to\Gamma$ such
    that $hq\nu:I\to X$ is a nullhomotopic closed path. There are only
    finitely many maximal vanishing paths and the folding process can
    be thought of as folding maximal vanishing paths one at a time. We
    now claim that $hq\nu$ has $D$-size $<10$ (i.e. $\diam
    Im(Dhq\nu)<10$) when $h$ is as in Proposition \ref{rs} (see Figure \ref{Vanishing_paths_figure} for an illustration). 
Indeed,
    $h'hq\nu$ is also a closed nullhomotopic path (where $h'$ is as in
    Proposition \ref{homotopies})
    and there is a
    homotopy of $h'hq\nu$ to $q\nu$ that moves the endpoints by $<3$
    measured by $D$. Thus the immersed path $q\nu$
    gets closed up to a nullhomotopic loop by a path of $D$-size $<6$,
    so it must itself have $D$-size $<6$, and so $hq\nu$ has $D$-size
    $<10$.

\begin{figure}[ht]
\includegraphics[width=\textwidth]{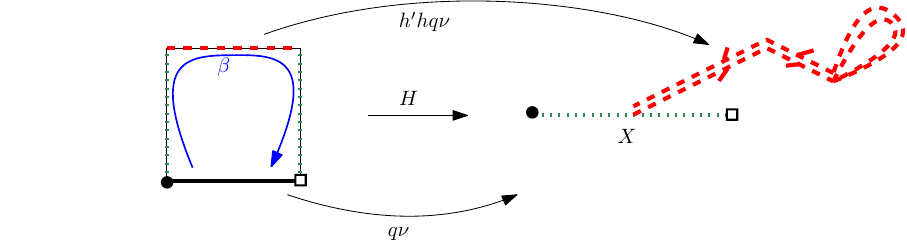}
\caption{$H:I\times I\to X$ denotes the homotopy from $q\nu$ to
  $h'hq\nu$. The path denoted by $\beta$ is mapped by $H$ to a path
  homotopic to $q\nu$. Since $H$ maps vertical segments $\{t\}\times
  I$ to paths whose images have $D$-length less than $3$, the dotted
  subpaths are mapped by $H$ to paths with $D$-length smaller than
  $3$. The dashed part of $\beta$ is null homotopic. Since $q\nu$ is
  immersed, its $D$-length is 
  $<6$.\label{Vanishing_paths_figure}}
\end{figure}

    We now observe that after folding $hq$ the tree
    components of the complement of the core have $D$-size
    $<10$. Indeed, choose any point $p\in\Gamma$. First fold all
    vanishing paths that do not contain $p$. After this, $p$ is still
    in the core. Finally, fold the remaining vanishing paths -- this
    operation changes only the neighborhood of $p$ of $D$-size $<10$.

    Now consider $(Dq)^{-1}[k+1,\infty)\subset\Gamma$. It is a
      disjoint union of (possibly degenerate) closed intervals in the
      interior of the separating arc of $\Gamma$. Form a new graph
      $\tilde\Gamma$ by attaching an edge $E_a$ to $\Gamma$ for every
      nondegenerate arc $a$ in this disjoint union, with $\partial
      E_a=\partial a$. Note that $q$ sends both endpoints of $a$ to
      the same vertex (at $D$-height $k+1$, i.e. distance $r_{k+1}$
      from the root vertex). Extend $q$ to an immersion
      $\tilde q:\tilde\Gamma\to X$ by sending each $E_a$ to a loop based at this
      vertex of combinatorial length $\leq 3$ (for example, one can
      send it either to the attached loop based at that vertex or to the loop
      of the form $dcd^{-1}$ where $d$ is an edge that increases the
      distance from the root and $c$ is the loop attached at the
      terminal vertex of $d$, see Figure \ref{GammaTildeFigure}). 
      Let $\Gamma'\subset\tilde\Gamma$ be the
      barbell obtained by deleting the interiors of the arcs $a$ as
      above, and let $q':\Gamma'\to X$ be the restriction of $\tilde q$.

\begin{figure}[ht]
\includegraphics[width=\textwidth]{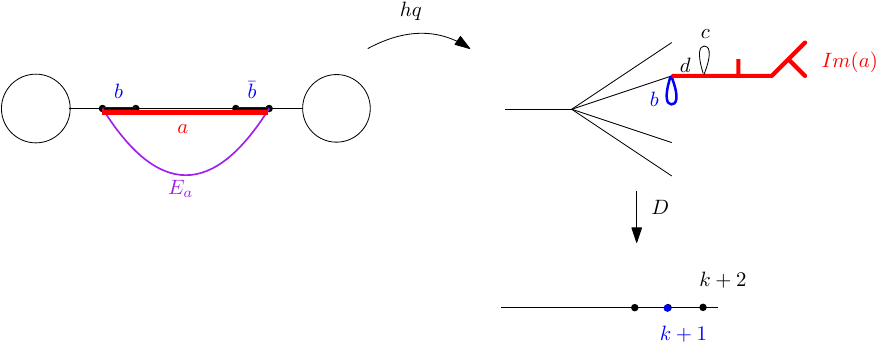}
\caption{The barbell graph on the left is the graph $\Gamma$. It is
  mapped via $h q$ to $X$. The graph $\tilde\Gamma$ is $\Gamma$ union
  the edge $E_a$. The arc labeled $a$ is mapped into
  $D^{-1}[k+1,\infty)$. Suppose the initial edge of of $a$ is mapped
    to the loop $b$ and the terminal edge is mapped to $\bar b$. Then
    in order to define $\tilde q$ on $\tilde\Gamma$ so that it will be
    an immersion we will let $E_a$ map to $dc\bar{d}$ where $d$ is the
    edge to the right of $b$, and $c$ is the one edge loop based at
    the endpoint of $d$.  \label{GammaTildeFigure}}
\end{figure}

      We now claim that $h\tilde q$ is homotopic into
      $D^{-1}(J_{n+1})$ for every $h\in H$.
      We can fold $h\tilde q$ by first folding $hq$, which produces a
      core graph with trees attached, and then adding the edges
      $E_a$. They could be attached to points in the attached trees,
      but all such attached trees have $D$-size $<10$ and map to
      $(k+1-10,k+1+10)\subset J_{n+1}$ by $D$. So after removing the
      attached trees to which no $E_a$'s are attached, the image is
      entirely contained in $D^{-1}(J_{n+1})$, which proves the
      claim.

      In particular, $hq':\Gamma'\to X$ is homotopic into
      $D^{-1}(J_{n+1})$, so $q'$ is also a witness to the fact that
      $x,x'$ lift to the same component of $\F^*(J_{n+1})$. By
      construction, the image of $q'$ does not exceed the $D$-height
      $k+2$, so $hq'$ does not exceed the $D$-height $k+4$, and we see
      that $hq'$ is contained (even without homotopies) in
      $D^{-1}(J_n)$. It follows that $hq'$ is homotopic into
      $D^{-1}(J_{n,n+1})$, and so the lifts of $x,x'$ in $\F^*(J)$ are
      in the same component.

To see that $\T^*$ has no loops, we note that if there were an
embedded loop then let $u \in V(\T^*)$ be the vertex of the loop of
maximal height. Since no two vertices of the same height are connected
then the loop gives us two vertices of the same height attached to $u$
which is a contradiction.
  \end{proof}

  Our next goal is to
  verify that $\T^*\to\T$ induces an isomorphism between the
  fundamental groups of these graphs of groups. Our method is to find a sequence
  of folds that converts $\T^*$ to $\T$. We will do this through an
  intermediate tree of groups $\T^*\to\T^{**}\to\T$. Only $\T^*$ will
  be $H$-invariant.

Recall the following folding moves on simplicial $G$-trees $T$
\cite{bf:bounding}. If $e_1$, $e_2$ are two oriented edges with the
common initial vertex $v$ such that $e_1\cup e_2$ embeds in the
quotient $T/G$, then we may construct a new $G$-tree $T'$ by
identifying $e_1$ and $e_2$ in an equivariant fashion, i.e. we
identify $g(e_1)$ and $g(e_2)$ for every $g\in G$. The stabilizer of
the new edge $e_1=e_2$ is the group generated by $Stab(e_1)$ and
$Stab(e_2)$, and similarly for the terminal vertices of $e_1$ and
$e_2$. The effect in the quotient graph is to fold the images of $e_1$
and $e_2$. This is called Move IA in \cite{bf:bounding}.

Similarly, suppose $e_1$, $e_2$ are two oriented edges with the common
initial vertex $v$, each edge embeds in the quotient $T/G$, but they
have the same images in $T/G$. This means that $g(e_1)=e_2$ for some
$g\in Stab(v)$, so $Stab(e_2)=g Stab(e_1) g^{-1}$. The equivariant
folding operation has the effect that the underlying quotient graph is
unchanged, but the stabilizer of $e_1=e_2$ is now the group generated
by $Stab(e_1)$ and $g$, and similarly for the terminal vertex. This is
called Move IIA, and we think of it as pulling the element $g\in
Stab(v)$ across the image edge to the terminal vertex and enlarging
the stabilizers by this $g$. In a similar way we can pull finitely
generated subgroups (or think of it as several Moves IIA performed in
sequence).

    Let $\T^{**}$ be the tree of groups obtained from $\T^*$ by
    folding each preimage of an edge to an edge, so that there is a
    morphism $\T^{**}\to\T$. This amounts to performing infinitely
    many Moves IA, but they are all independent and can be performed
    simultaneously. The resulting morphism  $\T^{**}\to\T$ is an isomorphism of
    underlying trees.

    It will be convenient to denote by $\mathcal T(e)$ the group
    associated to an edge $e$ of $\mathcal T$, and similarly for the
    vertices, and for the trees $\mathcal T^*$ and $\mathcal T^{**}$. 

    \begin{lemma}\label{5.9}
      After independent Moves IIA, the morphism $\T^{**}\to\T$ becomes
      an isomorphism of graphs of groups.
    \end{lemma}

    \begin{proof}
      The moves consist of pulling across an edge $e$ from an endpoint
      $w$ the subgroup $\T^{**}(w)\cap \T(e)$, simultaneously for all
      $(w,e)$. Since $J_{n,n+1} \smallsetminus J_n^- \subset J_{n+1}^-$ 
      then $\F(J_{n,n+1})$ is generated by elements of $\F(J_n^{-})$ and $\F(J_{n+1}^-)$ which are contained in $\F^*(J_n)$ and $\F^*(J_{n+1})$ respectively. 
Therefore the group $\T(e)$ is generated by elements in $\T^*(w), \T^*(v)$ for the endpoints $w,v$ of $e$. Thus by applying IIA moves we can promote $\T^{**}(e)$ to $\T(e)$. Similarly, $J_n \subset J_n^{-} \cup J_{n-1,n}^- \cup J_{n,n+1}^-$ hence $\T(w)$ is generated by elements in $\T^{**}(w)$ and $\{\T^{**}(e) \mid w \text{ is an endpoint of } e\}$. Therefore we can promote $\T^{**}(w)$ to $\T(w)$ using IIA moves.
    \end{proof}

    When $\mathcal Y$ is a locally finite graph of groups with all
    vertex and edge stabilizers finite rank free groups we define the
    {\it geometric realization} $GR(\mathcal Y)$. This is the
    2-complex constructed by taking a finite graph $\Gamma_w$ for
    every vertex $w$ so that $\pi_1(\Gamma_w)=\mathcal Y(w)$, and
    similarly taking a finite graph $\Gamma_e$ for every edge $e$ so
    that $\pi_1(\Gamma_e)=\mathcal Y(e)$, and gluing $\Gamma_e\times
    [0,1]$ according to
    inclusion homomorphisms. Up to a proper homotopy equivalence,
    $GR(\mathcal Y)$ is independent of the choices.
    From the lemmas above we see that the
    fundamental groups of graphs of groups $\T$, $\T^*$, $\T^{**}$ are
    all isomorphic to $\pi_1(X)$. We now upgrade this to proper
    homotopy equivalences of geometric realizations.

    \begin{lemma}\label{allsame}
      $X, GR(\T), GR(\T^*), GR(\T^{**})$ are all proper homotopy
      equivalent.
    \end{lemma}

    \begin{proof}
      $GR(\T)$ can be built as a subspace of $X\times [0,\infty)$:
        $$GR(\T)=\cup_{n=0}^\infty \big(D^{-1}(J_n)\times\{n\}\cup
        D^{-1}(J_n\cap J_{n+1})\times [n,n+1]\big)$$
        The map $GR(\T)\to X$ is the projection, and $X\to GR(\T)$ is
        the map $x\mapsto (x,\phi(x))$, where $\phi$ equals $n$ on
        $J_n\smallsetminus (J_{n-1}\cup J_{n+1})$ and is in $[n,n+1]$
        on $J_n\cap J_{n+1}$. These are each other's proper homotopy
        inverses by homotoping along the second coordinate.
  
        That $GR(\T^*)\to GR(\T^{**})\to GR(\T)$ are proper homotopy
        equivalences follows from the fact that Moves IA as well
        as IIA consisting of pulling finitely generated subgroups
        are proper homotopy equivalences on geometric realizations.
    \end{proof}

    To finish, we need the relative version of Nielsen Realization for
    graphs, proved by Hensel-Kielak.

    \begin{thm}[\cite{hensel-kielak}]
      Let $H<Out(F_n)$ be a finite subgroup and $\F$ an $H$-invariant
      free factor system. Suppose the action of $H$ on $\F$ is
      realized as a simplicial action of $H$ on a finite graph $\Gamma_0$ whose
      fundamental group is identified with $\F$ (so the components of
      $\Gamma_0$ correspond to the free factors in $\F$). Then there
      is a finite graph $\Gamma$, a simplicial action of $H$ on
      $\Gamma$, an $H$-equivariant embedding
      $\Gamma_0\hookrightarrow \Gamma$, and an identification
      $\pi_1(\Gamma)\cong F_n$ so that the induced $H\to Out(F_n)$ is
      the given embedding $H<Out(F_n)$.
    \end{thm}

    When $\F$ is empty, we have the (absolute) Nielsen Realization
    \cite{NR1,NR2,NR3,NR4}.

    To apply this, we note:

    \begin{lemma}
      For every vertex $w$ in $\T^*$ the incident edge groups form a free
      factor system in $\T^*(w)$.
    \end{lemma}

    \begin{proof}
      This is true for the tree $\T$ by construction. The statement
      then follows from the fact that intersections of free factor
      systems are free factor systems.
    \end{proof}

    Now we build a graph $Y$. We first construct graphs associated to
    the edges. Note that all orbits of edges are finite. For an
    edge $e$ of $\T^*$ choose a graph $\Gamma_e$ with
    $\pi_1(\Gamma_e)=\T^*(e)$ where $Stab_H(e)$ acts inducing the
    given action on $\T^*(e)$. Of course, $Stab_H(e)$ is a compact
    group, but the action on $\T^*(e)$ factors through a finite group,
    so we can apply the Nielsen Realization theorem.
    We associate the same graph to all edges in the orbit of $e$, with
    suitable identifications on $\pi_1$, so that $H$ now acts on the
    disjoint union of these graphs with the given action on $\pi_1$.

    Now consider a vertex $w$. We have that $Stab_H(w)$ acts on
    $\T^*(w)$ and this action factors through a finite group, which
    also acts on the free factor system defined by the incident
    edges. This action is realized by the action of $Stab_H(w)$ on the
    disjoint union of the graphs representing the edge spaces, so the
    Relative Nielsen Realization provides a finite graph $\Gamma_w$
    that contains this disjoint union and an extension of this
    action. Associate such graphs to the vertices equivariantly. The
    union along the subgraphs associated to the edges is the desired
    graph $Y$. Thus $H$ acts on $Y$ simplicially. The following lemma
    finishes the proof of the Main Theorem in the core graph case.

    \begin{lemma}
      There is a proper homotopy equivalence $Y\to X$ that commutes
      with the action of $H$.
    \end{lemma}

    \begin{proof}
      Using the same graphs to represent vertex and edge groups, the
      geometric realization $GR(\T^*)$, after collapsing the
      $I$-factors, becomes $Y$, and this is a proper homotopy
      equivalence. By composing with proper homotopy equivalences from
      Lemma \ref{allsame} we have $f:X\to Y$ and $g:Y\to X$, which are
      each other's inverses. If $h\in H$ then by construction $h:X\to
      X$ and $ghf:X\to X$ induce the same element of
      $Out(\pi_1(X))$. It then follows from Theorem \ref{id} applied
      to $ghf\cdot h^{-1}$ that they are properly homotopic.
    \end{proof}

\section{Proof for trees}\label{5}

We next prove Nielsen realization for trees.
  
  \begin{thm}
    Suppose the graph $X$ is a tree and let $H<\maps(X)$ be a compact
    subgroup. Then there is a tree $Y\simeq X$ where $H$ acts by
    simplicial isomorphisms.
  \end{thm}

  Note that by Corollary \ref{4.12} $\maps(X)=Homeo(\partial
  X)$. Fix a metric $d$ on $\partial X$.

  {\bf Step 1.} We replace $d$ by an $H$-invariant metric $d'$. Let $\nu$
  be a Haar measure on $H$ and define
  $$d'(p,q)=\int_H d(h(p),h(q))~ d\nu$$
  This is an $H$-invariant metric.
  We drop the prime and assume $d$ is $H$-invariant.

  {\bf Step 2.} We now build equivariant finite partitions of $\partial X$
  into clopen sets. Let $\epsilon>0$. Say $p,q\in \partial X$ are
  $\epsilon$-path connected if there is a sequence
  $p=z_0,z_1,\cdots,z_n=q$ so that $d(z_i,z_{i+1})<\epsilon$ for all
  $i=0,\cdots,n-1$. The equivalence classes form the desired
  partition $\mathcal P_{\epsilon}$. Note that if $\epsilon<\epsilon'$
  then $\mathcal P_\epsilon$ refines $\mathcal P_{\epsilon'}$ and if
  $\mathcal P$ is an arbitrary finite partition into clopen sets,
  there is $\epsilon>0$ so that $\mathcal P_{\epsilon}$ refines
  $\mathcal P$.

  {\bf Step 3.} Finally we build $Y$ as the {\it mapping telescope} of
  a sequence of partitions from Step 2. Fix a decreasing sequence
  $\epsilon_n\to 0$ with $n=1,2,\cdots$ and let $\mathcal
  P_n:=\mathcal P_{\epsilon_n}$. We also set $\mathcal P_0$ to be the
  trivial partition $\{\partial X\}$.  Since $\mathcal P_{n+1}$
  refines $\mathcal P_n$ we have a natural surjection $\mathcal
  P_{n+1}\to \mathcal P_{n}$ induced by inclusion of sets. Now let $Y$
  be the mapping telescope of this sequence. More concretely, the set
  of vertices is the disjoint union $\sqcup_{n=0}^\infty \mathcal
  P_n\times\{n\}$, and there is an edge from $P\times\{n+1\}$ to
      $Q\times\{n\}$ whenever $P\subseteq Q$ (here $P\in\mathcal
      P_{n+1}$ and $Q\in \mathcal P_n$). Then $Y$ is a tree and
      $\partial Y$ is naturally (and $H$-equivariantly) homeomorphic
      to $\partial X$ by the homeomorphism that sends a branch
      $(P_n)_n$ of $Y$ to the point $\cap_n P_n$ in $\partial X$. The
      theorem is now proved since we have natural 
      identifications
  $$\maps(X)=Homeo(\partial X)=Homeo(\partial Y)=\maps(Y)$$
  and $H$ acts simplicially on $Y$.

  \section{Proof in general}

Let $X$ be a locally finite graph which is not a tree and assume that
a compact group $H$ is acting on $X$ by proper homotopy
equivalences. The action then restricts to the core $X_g$ (see Lemma
\ref{restriction}) and by the special case of core graphs there is a
core graph $Y_g$, an action of $H$ by simplicial isomorphisms on
$Y_g$, and an $H$-equivariant proper homotopy equivalence $f:X_g\to
Y_g$. 

\begin{lemma}
  There is a locally finite graph $Y\supseteq Y_g$ and a proper
  homotopy equivalence $X\to Y$ that extends $f$.
\end{lemma}

\begin{proof}
  Form the mapping cylinder $M=X_g\times I\sqcup Y_g/x\sim f(x)$ of
  $f$. Since $f$ is a proper homotopy equivalence, both 0 and 1-levels
  of $M$ (which can be identified with $X_g$ and $Y_g$) are proper strong
  deformation retracts of $M$. For $Y_g$ this can be seen by deforming
  along the mapping cylinder lines. For $X_g$, without the word
  ``proper'', this is a theorem of Ralph Fox \cite{fox}, see also
  \cite{fuchs}, but their proofs work just as well in the proper
  category. The statement can also be deduced from the 
  Whitehead theorem, see \cite{hatcher}, and
  \cite{farrell-taylor-wagoner} for the proper version. Now $X$ is
  obtained from $X_g$ by attaching trees $T_v$ along vertices $v\in
  X_g$. Attach products $T_v\times I$ to $M$ along the natural copies
  of $\{v\}\times I$ to obtain a space $Z$ and note that both $X$ and
  the space $Y$ (obtained from $Y_g$ by attaching trees $T_v$ along
  $f(v)$) are proper strong deformation retracts of $Z$ and this gives
  the desired proper homotopy equivalence $X\to Y$.
\end{proof}

We will now revert to the original notation and simply assume that $H$
is acting by simplicial isomorphisms on $X_g$.

  By the {\it convex hull} of a nonempty subset of a simplicial tree we mean
  the smallest simplicial subtree that contains the set. The following
  fixed point fact is well known.

  \begin{lemma}\label{fixed0}
    Suppose a compact group $H$ acts continuously on a simplicial
    tree. Then $H$ fixes a point in the convex hull of any orbit.
  \end{lemma}

  \begin{proof}
    The convex hull is $H$-invariant and it is a tree of finite
    diameter. Iteratively remove all edges that contain a valence 1
    vertex until the tree that's left is either a single vertex or a
    single edge. This vertex or the midpoint of the edge is then fixed
    by $H$.
  \end{proof}

  We will now use this fact to prove the following fixed point
  theorem, which is really the heart of the argument in this case.

  \begin{lemma}\label{fixed}
    Suppose $H$ fixes a point $\beta$ in $DX=\partial X\smallsetminus
    \partial X_g$. Then $H$ fixes a point $\rho(\beta)$
    in $X_g$ and there is a ray (called the {\it Nielsen ray}) $r$ from
    $\rho(\beta)$ to $\beta$ such that $h(r)$ and $r$ are properly
    homotopic rel $\rho(\beta)$ for every $h\in H$.
  \end{lemma}

  \begin{proof}
    Let $\tilde X$ be the universal cover of $X$. Let $r$ be a ray in
    $X$ converging to $\beta$. The deck group acts simply transitively on
    the set of lifts of $r$ and distinct lifts are not asymptotic and
    hence not properly homotopic. Choose one such lift $\tilde r$.
    Every $h\in H$ has a unique lift to $\tilde X$ that fixes the
    asymptotic class of rays $[\tilde r]$ 
    and the set of these lifts defines an action of $H$ on
    $\tilde X$ by proper homotopy equivalences. We will prove that the action is continuous in the next paragraph.  The lifted group $H$ preserves the preimage $\tilde X_g$ of
    $X_g$, which is a tree, and this defines an action of $H$ on
    $\tilde X_g$. By Lemma \ref{fixed0} it fixes
    a point $z$. The image of $z$ in $X_g$ is the desired fixed
    point and the image of the ray that starts at $z$ and is
    asymptotic to $\tilde r$ is the Nielsen ray.

    The action
    is continuous: if $h\in H$ is close to the identity, we can choose
    a representative in its proper homotopy class that fixes a large
    compact set $K\subset X$ as well as the ray $r$, and preserves the
    complementary components of $K$. We can also arrange that $K\cup
    r$ is connected. Then the lift of $h$ to $\tilde X$ will fix the
    preimage $\tilde K$ and will preserve its complementary
    components. Since $K$ can be chosen so that $\tilde K$ contains
    any given compact set, the lift of $h$ will be close to the
    identity.
     \end{proof}

    Let $d$ be an $H$-invariant metric on $\partial X$ (see Step 1 in
    Section \ref{5}) and let $\mathcal P_\epsilon$ be the partition of
    $\p X$ as in Step 2 in Section \ref{5}. Again fix a decreasing
    sequence $\epsilon_n\to 0$ and set $\mathcal P_n:=\mathcal
    P_{\epsilon_n}$. Let $\pi':X\cup DX\to X_g$ denote the nearest point
    projection (this is not equivariant).

    Fix an $H$-equivariant exhaustion $\emptyset=K_0\subset K_1\subset
    K_2\subset\cdots$ of $X_g$ by finite connected subgraphs so that
    if $\beta\in X\cup DX$ and $\pi'(\beta)\not\in K_{i+1}$ then
    $\pi'(h(\beta))\not\in K_i$ for every $h\in H$.

    Call an element $P\in\mathcal P_n$ {\it good} if the following
    holds:
    \begin{itemize}
    \item $P\subset DX$,
      \item $\pi'(P)$ is a point, 
    \item $Stab_H(P)$ fixes a point $\rho(P)$ in $X_g$; moreover,
      if $\pi'(P)$ is disjoint from $K_{i+1}$ then 
      $\rho(P)$ and $\pi'(P)$ are in the same
      component of $X_g\smallsetminus K_i$,
      \item for every $x\in P$ there is a ray $r_x$ from $\rho(P)$ to $x$ so
        that all these rays (for all $x\in P$) agree along $X_g$ and
        further they are permuted up to proper homotopy by $Stab_H(P)$.
    \end{itemize}

So in particular $r_x$ is a Nielsen ray with respect to
$Stab_H(x)<Stab_H(P)$. We will also call
        the rays $r_x$ {\it Nielsen rays}.

\begin{lemma}
  For every $\beta\in DX$ there is $n_0$ so that for every $n\geq n_0$ the
  element $P\in \mathcal P_n$ containing $\beta$ is good.
\end{lemma}

\begin{proof}
  We first observe that $Stab_H(\beta)$ fixes a point in $X_g$ by
  applying Lemma \ref{fixed} to the induced action of $Stab_H(\beta)$
  on the graph $X_g^*=X_g\cup \rho_\beta$ (see Lemma
  \ref{restriction}). By our assumption on the exhaustion, if
  $\pi'(\beta)$ misses $K_{i+1}$ then the action restricts to the
  complementary component of $K_i$ that contains $\pi'(\beta)$, so in
  this case the fixed point $\rho(\beta)$ can be found there. Now notice that the
  stabilizer of a point in $X_g$ is a clopen subgroup of $H$, so when
  $n$ is large the stabilizer of $P_n\in \mathcal P_n$ that contains
  $\beta$ will also fix the same point. (Since $H$ permutes the
  partition elements in $\mathcal P_n$, $Stab_H(\beta)$ will leave $P_n$
  invariant and we see that $Stab_H(\beta)=\cap_n Stab_H(P_n)$ is the
  intersection of clopen subgroups. By compactness we have
  $Stab_H(P_n)\subseteq Stab_H(\rho(\beta))$ for large $n$.)
  We will of course also have
  $P\subset DX$, $\pi'(P)$ is a point, and $h(\ell)\cap X_g=\emptyset$
  for every line $\ell$ joining two points of $P$ and every $h\in H$.
\end{proof}

Now we construct an $H$-equivariant cover $\mathcal N$ by pairwise
disjoint good partition elements. Say an $H$-orbit in $\mathcal P_n$
(which is finite) is {\it good} if every (any) element in it is good. Then
let $\mathcal N$ consist of good orbits in $\mathcal P_1$ as well as
those good orbits in $\mathcal P_n$, $n=2,3,\cdots$ whose union is not
contained in the union of any good orbit in $\mathcal P_{n-1}$.
Define an equivariant map $\rho:\mathcal N\to X_g$ by letting $\rho$
be as in the definition of a good partition element on a
representative of the orbit, and then extend it equivariantly. Thus we
still have the Nielsen rays for all elements of $\mathcal N$.

We now construct a graph $Y$ by attaching trees to $Y_g=X_g$. For
every $N\in \mathcal N$ we build a tree $T_N$ as in Step 3 of Section
\ref{5} for $Stab_H(N)$, namely the mapping telescope with base vertex $N$ and the
other vertices all the partition elements contained in $N$. We
identify $\partial T_N$ with $N$. We then attach $T_N$ to $X_g$ by
identifying the base vertex $N$ with the point $\rho(N)\in X_g$. Doing
this for all $N\in \mathcal N$ produces the desired graph $Y$. By
construction $H$ acts on $Y$ by simplicial isomorphisms.

\begin{lemma}
  There is a proper homotopy equivalence $F:Y\to X$ such that
  \begin{enumerate}[(a)]
    \item $F$ is identity on
      $X_g$ and on $DX=DY$,
      \item $F$ sends the rays in $T_N$ based at $N$ to the Nielsen rays $r_x$
        from $\rho(N)$ to
        $\partial N$ preserving the endpoints,
                  \item
            $F$ is $H$-equivariant.
            \end{enumerate}
\end{lemma}

\begin{proof}
  The map $F$ is uniquely defined on each $T_N$ by (a)-(c). That this
  map is proper as a map $Y\to X$ follows from the fact that if
  $N_i\in \mathcal N$ converge to $\beta\in\partial X_g$, then
  $\rho(N_i)\to\beta$. Thus $F$ is a proper homotopy equivalence by
  Corollary \ref{criterion2}.

  Finally we argue $H$-equivariance. Denote by $F'$ the proper
  homotopy inverse of $F$ which is identity on $X_g$. If $h\in H$
  consider $F'hF\cdot h^{-1}:Y\to Y$. This is identity on $X_g$ and on
  $\p X$. By Corollary \ref{criterion} it suffices to argue that this
  map preserves oriented loops and lines connecting points of
  $DX$. For loops this is clear since the map is identity on $X_g$. It
  also preserves lines joining points of $DX$ since such lines can be written as a concatenation
  $r^{-1}sr'$ where $r,r'$ are Nielsen rays and $s$ is a segment in
  $X_g$. Finally, it preserves lines that connect distinct points of
  some $N\in \mathcal N$. 
\end{proof}

This finishes the proof of the Main Theorem.

\bibliographystyle{amsplain}

\end{document}